\documentclass[11pt]{article}
\usepackage{amssymb}
\usepackage{amsmath} 
\usepackage{mathrsfs}
\usepackage{graphics}
\usepackage{graphicx}
\usepackage{xcolor}
\usepackage{subfigure}
\usepackage[T1]{fontenc}
\usepackage{latexsym,amssymb,amsmath,amsfonts,amsthm}\usepackage{txfonts}
\newcommand{\be}{\begin{eqnarray}}
\newcommand{\ben}{\begin{eqnarray*}}
\newcommand{\en}{\end{eqnarray}}
\newcommand{\enn}{\end{eqnarray*}}

\newtheorem{theorem}{Theorem}[section]
\newtheorem{lemma}{Lemma}[section]
\newtheorem{prp}[theorem]{Proposition}
\newtheorem{thm}[theorem]{Theorem}
\newtheorem{cor}[theorem]{Corollary}
\newtheorem{dfn}{Definition}[section]

\newtheorem{remark}{Remark}

\definecolor{rr}{rgb}{0,0,0}
\definecolor{r}{rgb}{0,0,0}
\begin{document}
\renewcommand{\theequation}{\arabic{section}.\arabic{equation}}
\begin{titlepage}
\title{\bf  On the small time asymptotics of scalar stochastic conservation laws
}
\author{Zhao Dong$^{1,2}$\ \ Rangrang Zhang$^{3,}\footnote{Corresponding author.}$\\
{\small $^1$
RCSDS, Academy of Mathematics and Systems Science, Chinese Academy of Sciences, Beijing 100190, China.} \\
{\small $^2$ School of Mathematical Sciences, University of Chinese Academy of Sciences.}\\
{\small $^3$ School of Mathematics and Statistics,
Beijing Institute of Technology, Beijing, 100081, China.}\\
({\sf dzhao@amt.ac.cn}, {\sf rrzhang@amss.ac.cn} )}
\date{}
\end{titlepage}
\maketitle

\noindent\textbf{Abstract}:
In this paper, we established a small time large deviation principles for scalar stochastic conservation laws driven by multiplicative noise. The doubling variables method plays a key role.



\section{Introduction}
In this paper, we investigate the small time asymptotics of the first-order scalar conservation laws with stochastic forcing.
 Precisely, fix any
$T>0$ and let $(\Omega,\mathcal{F},\mathbb{P},\{\mathcal{F}_t\}_{t\in
[0,T]},(\{\beta_k(t)\}_{t\in[0,T]})_{k\in\mathbb{N}})$ be a stochastic basis. Without loss of generality, here the filtration $\{\mathcal{F}_t\}_{t\in [0,T]}$ is assumed to be complete and $\{\beta_k(t)\}_{t\in[0,T]}(k \in\mathbb{N})$ are one-dimensional i.i.d. real-valued  $\{\mathcal{F}_t\}_{t\in [0,T]}-$Wiener processes. The symbol $\mathbb{E}$ denotes the expectation with respect to $\mathbb{P}$.
For any fixed $N\in\mathbb{N}$, let $\mathbb{T}^N\subset\mathbb{R}^N$ be the $N-$dimensional torus with the periodic length to be $1$.
We are concerned with the following Cauchy problem for the scalar conservation laws with stochastic forcing
\begin{eqnarray}\label{P-19}
\left\{
  \begin{array}{ll}
  du+div(A(u))dt=\Phi(u) dW(t) \quad {\rm{in}}\ \mathbb{T}^N\times (0,T),\\
u(\cdot,0)=\eta(\cdot) \quad {\rm{on}}\ \mathbb{T}^N,
  \end{array}
\right.
\end{eqnarray}
where $u:(\omega,x,t)\in\Omega\times\mathbb{T}^N\times[0,T]\mapsto u(\omega,x,t):=u(x,t)\in\mathbb{R}$ is a random field, the flux function $A:\mathbb{R}\to\mathbb{R}^N$ and the coefficient $\Phi:\mathbb{R}\to\mathbb{R}$ are measurable and fulfill certain conditions (see Section 2 in below), and $W$ is a cylindrical Wiener process defined on a given (separable) Hilbert space $U$ with
the form $W(t)=\sum_{k\geq 1}\beta_k(t) e_k,t\in[0,T]$, where {\color{r}$\{e_k\}_{k\geq 1}$ is an orthonormal base} of the Hilbert space $U$. Moreover, the initial value {\color{r}$\eta\in L^{\infty}(\mathbb{T}^N)$} is a deterministic function.

When $\Phi\equiv0$, the system (\ref{P-19}) is reduced to the deterministic scalar conservation law, which is fundamental to our understanding of the space-time evolution laws of interesting physical quantities. For more background on this model, we refer the readers to
the monograph \cite{Dafermos}, the work of {\color{r}Ammar,
Wittbold and Carrillo} \cite{K-P-J} and references therein. As we know, the Cauchy problem
for the deterministic first-order PDE (\ref{P-19}) does not admit any (global) smooth solutions, but there exist infinitely many weak solutions to the deterministic Cauchy problem. To solve the problem of non-uniqueness, an additional entropy condition was added to identify the physical weak solution. Under this condition,
the notion of entropy solutions for the deterministic first-order scalar conservation laws was introduced by Kru\v{z}kov \cite{K-69,K-70}.
The kinetic formulation of weak entropy solution of the Cauchy problem for a general multi-dimensional scalar conservation laws (also called the kinetic system), was derived by Lions, Perthame and Tadmor in \cite{L-P-T}. The authors of \cite{L-P-T} also discussed the relationship between entropy solutions and the kinetic system.

Adding a stochastic forcing (i.e., a noise) to this physical model is quite natural as it either represents an external random perturbation or gives a remedy for lack of (empirical) knowledge of certain involved physical parameters.
Along with the great successful developments of deterministic
scalar conservation laws, the random situation has also been developed rapidly.
 For example, in \cite{K}, Kim studied the Cauchy problem for the scalar stochastic conservation laws (\ref{P-19}) driven by additive noise. Later, these results were extended to the multi-dimensional Dirichlet problem with additive noise by Vallet and Wittbold in \cite{V-W}. The authors of \cite{V-W} succeed to show the existence and uniqueness of the stochastic entropy solutions by utilising the vanishing viscosity method, Young measure techniques, and Kru\v{z}kov doubling variables technique. Concerning the multiplicative noise, for the Cauchy problem over the whole spatial space, Feng and Nualart \cite{F-N} introduced a notion of strong entropy solutions to prove the uniqueness of the entropy solution. Moreover, the authors in \cite{F-N} established the existence of strong entropy solutions in one dimensional case by using the vanishing viscosity and compensated compactness method.
 Recently, Debussche and Vovelle \cite{D-V-1} proved the existence and uniqueness of kinetic solution to the Cauchy problem for (\ref{P-19}) in any dimension by utilizing a kinetic formulation developed by Lions, Perthame and Tadmor for deterministic first-order scalar conservation laws in \cite{L-P-T}. Due to the equivalence between kinetic formulation and entropy solution, the existence and uniqueness of the entropy solutions were obtained in \cite{D-V-1}. {\color{r}It is worth mentioning that \cite{D-V-1} is the starting point of the present paper.}
 In addition, the long-time behavior of the first-order scalar conservation laws has also attracted a lot of interests. For example, Debussche and Vovelle established the existence and uniqueness of invariant measures of scalar conservation laws driven by
  additive stochastic forcing in \cite{D-V-2}. Concretely, for sub-cubic fluxes, the authors of \cite{D-V-2} show the existence of an invariant measure, and for sub-quadratic fluxes, they proved the uniqueness of the invariant measure.
Recently, combining techniques used in the context of kinetic solutions as well as new results on large deviations, Dong et al. \cite{DWZZ} established Freidlin-Wentzell's type large deviation principles (LDP) for the kinetic solution to the scalar stochastic conservative laws.

\smallskip
 The purpose of this paper is to investigate the small time LDP of the kinetic solution to the scalar stochastic conservation laws, which describes the behaviors of the solution at a very small time. Specifically, we focus on the limiting behavior of the kinetic solution to the scalar stochastic conservation laws in a time interval $[0,t]$ as $t$ goes to zero. An important motivation for such a problem comes from Varadhan identity
\begin{eqnarray*}
\lim_{t\rightarrow 0}2t\log \mathbb{P}\big(u(0)\in B,\ u(t)\in C\big)=-d^2(B,C),
\end{eqnarray*}
where $u$ is the kinetic solution to the scalar stochastic conservation laws and $d$ is an appropriate Riemann distance associated with the diffusion generated by $u$.
The mathematical study of the small time LDP for finite dimensional processes was initiated by Varadhan \cite{V}. Since then, the cases for the infinite dimensional diffusion processes were extensively studied (see \cite{A-K, A-Z, F-Z, H-R, ZTS} and the references therein).
On the other hand, many researchers have also studied the small time LDP for infinite dimensional stochastic partial differential equations. For instance, Xu and Zhang \cite{X-Z} established the small time LDP of 2D Navier-Stokes equations in the state space $C([0,T];H)$. Dong and Zhang \cite{D-Z} proved the small time LDP of 3D stochastic primitive equations in the state space  $C([0,T];H^1)$. In this paper, we will prove that the small time LDP of the kinetic solution to the scalar stochastic conservation laws holds in the space $L^1([0,T];L^1(\mathbb{T}^N))$. To our knowledge, the present paper is the first work towards proving the small time LDP directly for the kinetic solution to the scalar stochastic conservation laws.
Due to the fact that the kinetic solutions are living in a rather irregular space, we will use the doubling variables method as in the work of Debussche and Vovelle \cite{D-V-1}. Our new contribution is the estimation of martingale terms and error terms, which are highly nontrivial.

The rest of the paper is organized as follows. In Section 2, we recall the mathematical formulation of scalar stochastic conservation laws. In Section 3, we introduce the small time asymptotics and state our main result. Section 4 is devoted to the proof of exponential equivalence.

\section{Framework}
In the following, we will follow closely the framework of \cite{D-V-1} to
 introduce some notations.
Let $\|\cdot\|_{L^p}$ denote the norm of usual Lebesgue space $L^p(\mathbb{T}^N)$ for $p\in [1,\infty]$. In particular, set $H=L^2(\mathbb{T}^N)$ with the corresponding norm $\|\cdot\|_H$.
  $C_b$ represents the space of bounded, continuous functions and $C^1_b$ stands for the space of bounded, continuously differentiable functions having bounded first order derivative. Define the function $f(x,t,\xi):=I_{u(x,t)>\xi}$, which is the characteristic function of the subgraph of $u$. We write $f:=I_{u>\xi}$ for short.
 Moreover, denote by the brackets $\langle\cdot,\cdot\rangle$ the duality between $C^{\infty}_c(\mathbb{T}^N\times \mathbb{R})$ and the space of distributions over $\mathbb{T}^N\times \mathbb{R}$. In what follows, with a slight abuse of the notation $\langle\cdot,\cdot\rangle$, we denote the following integral by
\[
\langle F, G \rangle:=\int_{\mathbb{T}^N}\int_{\mathbb{R}}F(x,\xi)G(x,\xi)dxd\xi, \quad F\in L^p(\mathbb{T}^N\times \mathbb{R}), G\in L^q(\mathbb{T}^N\times \mathbb{R}),
\]
where $1\leq p\leq +\infty$, $q:=\frac{p}{p-1}$ is the conjugate exponent of $p$. In particular, when $p=1$, we set $q=\infty$ by convention. For a measure $m$ on the Borel measurable space $\mathbb{T}^N\times[0,T]\times \mathbb{R}$, the shorthand $m(\phi)$ is defined by
\[
m(\phi):=\langle m, \phi \rangle([0,T]):=\int_{\mathbb{T}^N\times[0,T]\times \mathbb{R}}\phi(x,t,\xi)dm(x,t,\xi), \quad  \phi\in C_b(\mathbb{T}^N\times[0,T]\times \mathbb{R}).
\]
In the sequel,  the notation $a\lesssim b$ for $a,b\in \mathbb{R}$  means that $a\leq \mathcal{D}b$ for some constant $\mathcal{D}> 0$ independent of any parameters.

\subsection{Hypotheses}
For the flux function $A$ and the coefficient $\Phi$ of (\ref{P-19}), we assume that
\begin{description}
  \item[\textbf{Hypothesis H}] The flux function $A$ belongs to $C^2(\mathbb{R};\mathbb{R}^N)$ and its derivative $a:=A'$ is polynomial growth with degree $q_0>1$. That is, there exists a constant $\mathcal{N}(q_0)\geq0$ such that
     \begin{eqnarray}\label{qeq-22}
    |a(\xi)|\leq \mathcal{N}(q_0)(1+|\xi|^{q_0}),\quad |a(\xi)-a(\zeta)|\leq \Upsilon(\xi,\zeta)|\xi-\zeta|,
      \end{eqnarray}
      where $\Upsilon(\xi,\zeta):=\mathcal{N}(q_0)(1+|\xi|^{q_0-1}+|\zeta|^{q_0-1})$.

      For each $u\in \mathbb{R}$, the map $\Phi(u): U\rightarrow H$ is defined by $\Phi(u) e_k=g_k(\cdot, u)$, where $(e_k)_{k\geq 1}$ is the orthonormal base in the Hilbert space $U$ and each $g_k(\cdot,u)$ is a regular function on $\mathbb{T}^N$.
      More precisely, we assume that $g_k\in C(\mathbb{T}^N\times \mathbb{R})$ satisfies the following bounds
\begin{eqnarray}\label{e-5}
|g_k(x,u)|&\leq& C^0_k(1+|u|), \quad \sum_{k\geq 1}|C^0_k|^2\leq \frac{D_0}{2},\\
\label{e-6}
|g_k(x,u)-g_k(y,v)|&\leq& C^1_k(|x-y|+|u-v|),\quad \sum_{k\geq 1}|C^1_{k}|^2\leq \frac{D_1}{2},
\end{eqnarray}
for $x, y\in \mathbb{T}^N, u,v\in \mathbb{R}$, where $C^0_k, C^1_k, D_0, D_1$ are positive constants.
\end{description}
From (\ref{e-5}) and (\ref{e-6}), we deduce that
\begin{eqnarray}\label{equ-28}
G^2(x,u)&:=&\sum_{k\geq 1}|g_k(x,u)|^2\leq D_0(1+|u|^2),\\
\label{equ-29}
\sum_{k\geq 1}|g_k(x,u)-g_k(y,v)|^2&\leq& D_1\Big(|x-y|^2+|u-v|^2\Big).
\end{eqnarray}
Based on the above notations, equation (\ref{P-19}) can be rewritten as
\begin{eqnarray}\label{P-19-1}
\left\{
  \begin{array}{ll}
  du(t,x)+div A(u(t,x))dt=\sum_{k\geq 1}g_k(x,u(t,x)) d\beta_k(t) \quad {\rm{in}}\ \mathbb{T}^N\times (0,T],\\
u(\cdot,0)=\eta(\cdot) \quad {\rm{on}}\ \mathbb{T}^N.
  \end{array}
\right.
\end{eqnarray}

\subsection{Kinetic solution}
Let us recall the notion of a kinetic solution to equation (\ref{P-19}) from \cite{D-V-1}. Keeping in mind that we are working on the stochastic basis $(\Omega,\mathcal{F},\mathbb{P},\{\mathcal{F}_t\}_{t\in [0,T]},(\beta_k(t))_{k\in\mathbb{N}})$.
\begin{dfn}(Kinetic measure)\label{dfn-3}
 A map $m$ from $\Omega$ to the set of non-negative, finite measures over $\mathbb{T}^N\times [0,T]\times\mathbb{R}$ is said to be a kinetic measure, if
\begin{description}
  \item[1.] $ m $ is measurable, that is, for each $\phi\in C_b(\mathbb{T}^N\times [0,T]\times \mathbb{R}), \langle m, \phi \rangle: \Omega\rightarrow \mathbb{R}$ is measurable,
  \item[2.] $m$ vanishes for large $\xi$, i.e.,
\begin{eqnarray}\label{equ-37}
\lim_{R\rightarrow +\infty}\mathbb{E}[m(\mathbb{T}^N\times [0,T]\times B^c_R)]=0,
\end{eqnarray}
where $B^c_R:=\{\xi\in \mathbb{R}, |\xi|\geq R\}$,
  \item[3.] for every $\phi\in C_b(\mathbb{T}^N\times \mathbb{R})$, the process
\[
(\omega,t)\in\Omega\times[0,T]\mapsto \langle m,\phi\rangle([0,t]):= \int_{\mathbb{T}^N\times [0,t]\times \mathbb{R}}\phi(x,\xi)dm(x,s,\xi)\in\mathbb{R}
\]
is predictable.
\end{description}
\end{dfn}
\begin{remark}\label{r-1}
  For any $\phi\in C_b(\mathbb{T}^N\times \mathbb{R})$ and kinetic measure $m$, define $A_t:=\langle m, \phi\rangle([0,t]),$
   then a.s., $t\mapsto A_t$ is a right continuous function of finite variation. Moreover, the function $A$ has left limits in any $t\in (0,T]$. We write $A_{t^-}=\lim_{s\uparrow t}A_s$ and set $A_{0^-}=0$.
  As a result, $A_{t^-}=\langle m, \phi\rangle([0,t))$, which is c\`{a}gl\`{a}d (left continuous with right limits).

\end{remark}

\begin{dfn}(Kinetic solution)\label{dfn-1}
Let $\eta\in L^{\infty}(\mathbb{T}^N)$. A measurable function $u: \mathbb{T}^N\times [0,T]\times\Omega\rightarrow \mathbb{R}$ is called a kinetic solution to (\ref{P-19}) with initial datum $\eta$, if
\begin{description}
  \item[1.] $(u(t))_{t\in[0,T]}$ is predictable,
  \item[2.] for any $p\geq1$, there exists $C_p\geq0$ such that
\begin{eqnarray}\label{qq-8}
\mathbb{E}\left(\underset{0\leq t\leq T}{{\rm{ess\sup}}}\ \|u(t)\|^p_{L^p(\mathbb{T}^N)}\right)\leq C_p,
\end{eqnarray}
\item[3.] there exists a kinetic measure $m$ such that $f:= I_{u>\xi}$ satisfies: for all $\varphi\in C^1_c(\mathbb{T}^N\times [0,T)\times \mathbb{R})$,
\begin{eqnarray}\notag
&&\int^T_0\langle f(t), \partial_t \varphi(t)\rangle dt+\langle f_0, \varphi(0)\rangle +\int^T_0\langle f(t), a(\xi)\cdot \nabla \varphi (t)\rangle dt\\ \notag
&=& -\sum_{k\geq 1}\int^T_0\int_{\mathbb{T}^N} g_k(x, u(t,x))\varphi (x,t,u(x,t))dxd\beta_k(t) \\ \label{P-21}
&& -\frac{1}{2}\sum_{k\geq 1}\int^T_0\int_{\mathbb{T}^N}\partial_{\xi}\varphi (x,t,u(x,t))G^2(x,u(t,x))dxdt+ m(\partial_{\xi} \varphi), \ a.s. ,
\end{eqnarray}
where $f_0=I_{\eta>\xi}$, $u(t)=u(\cdot,t,\cdot)$ and $G^2=\sum^{\infty}_{k=1}|g_k|^2$.
\end{description}
\end{dfn}
Let $u$ be a kinetic solution to (\ref{P-19}) and $f=I_{u>\xi}$. We use $\bar{f}:=1-f$ to denote its conjugate function. Define $\Lambda_f:=f-I_{0>\xi}$, which can be viewed as a correction to $f$. Note that $\Lambda_f$ is integrable on $\mathbb{T}^N\times [0,T]\times\Omega$ if $u$ is. In addition,
it is shown in \cite{D-V-1} that almost surely, the function $f=I_{u>\xi}$ admits left and right weak limits at any point $t\in[0,T]$.
\begin{prp}(\cite{D-V-1}, Left and right weak limits)\label{prp-3} Let $f=I_{u>\xi}$ satisfy (\ref{P-21}) with initial value $f_0=I_{\eta>\xi}$. Then $f$ admits, almost surely, left and right limits respectively at every point $t\in [0,T]$. More precisely, for any  $t\in [0,T]$, there exist kinetic functions $f^{t\pm}$ on $\Omega\times \mathbb{T}^N\times \mathbb{R}$ such that $\mathbb{P}-$a.s.
\begin{eqnarray*}
\langle f(t-r),\varphi\rangle\rightarrow \langle f^{t-},\varphi\rangle, \quad
\langle f(t+r),\varphi\rangle\rightarrow \langle f^{t+},\varphi\rangle
\end{eqnarray*}
as $r\rightarrow 0$ for all $\varphi\in C^1_c(\mathbb{T}^N\times \mathbb{R})$. Moreover, almost surely,
\[
\langle f^{t+}-f^{t-}, \varphi\rangle=-\int_{\mathbb{T}^N\times[0,T]\times \mathbb{R}}\partial_{\xi}\varphi(x,\xi)I_{\{t\}}(s)dm(x,s,\xi).
\]
In particular, almost surely, the set of $t\in [0,T]$ fulfilling $f^{t+}\neq f^{t-}$ is countable.
\end{prp}
For the above function $f=I_{u>\xi}$, define $f^{\pm}$ by $f^{\pm}(t)=f^{t \pm}$, $t\in [0,T]$. Since we are dealing with the filtration associated to Brownian motion, both $f^{+}$ and $f^{-}$ are  clearly predictable as well. Also $f=f^+=f^-$ almost everywhere in time and we can take any of them in an integral with respect to the Lebesgue measure or in a stochastic integral. However, if the integral is with respect to a measure, typically a kinetic measure in this article, the integral is not well-defined for $f$ and may differ if one chooses $f^+ $ or $f^-$.

At the end of this part, we mention that with the aid of Proposition \ref{prp-3}, the following result was verified by \cite{D-V-1}.
 \begin{lemma}\label{lem-1}
 The weak form (\ref{P-21}) satisfied by $f=I_{u>\xi}$ can be strengthened to be weak only with $x$ and $\xi$. Concretely, for all $t\in [0,T)$ and $\varphi\in C^1_c(\mathbb{T}^N\times \mathbb{R})$,
 \begin{eqnarray}\notag
\langle f^+(t),\varphi\rangle&=&\langle f_{0}, \varphi\rangle+\int^t_0\langle f(s), a(\xi)\cdot \nabla \varphi\rangle ds\\
\notag
&&+\sum_{k\geq 1}\int^t_0\int_{\mathbb{T}^N}\int_{\mathbb{R}}g_k(x,\xi)\varphi(x,\xi)d\nu_{x,s}(\xi)dxd\beta_k(s)\\
\label{qq-17}
&& +\frac{1}{2}\int^t_0\int_{\mathbb{T}^N}\int_{\mathbb{R}}\partial_{\xi}\varphi(x,\xi)G^2(x,\xi)d\nu_{x,s}(\xi)dxds- \langle m,\partial_{\xi} \varphi\rangle([0,t]), \quad a.s.,
\end{eqnarray}
with $\nu_{x,s}=-\partial_{\xi}f=\delta_{u(x,s)=\xi}$ and we set $f^+(T)=f(T)$.
 \end{lemma}

\begin{remark}  By making modifications, we have for all $t\in (0,T]$ and $\varphi\in C^1_c(\mathbb{T}^N\times \mathbb{R})$,
   \begin{eqnarray}\notag
\langle f^-(t),\varphi\rangle&=&\langle f_{0}, \varphi\rangle+\int^t_0\langle f(s), a(\xi)\cdot \nabla \varphi\rangle ds\\
\notag
&&+\sum_{k\geq 1}\int^t_0\int_{\mathbb{T}^N}\int_{\mathbb{R}}g_k(x,\xi)\varphi(x,\xi)d\nu_{x,s}(\xi)dxd\beta_k(s)\\
\label{e-80}
&& +\frac{1}{2}\int^t_0\int_{\mathbb{T}^N}\int_{\mathbb{R}}\partial_{\xi}\varphi(x,\xi)G^2(x,\xi)d\nu_{x,s}(\xi)dxds- \langle m,\partial_{\xi} \varphi\rangle([0,t)), \quad a.s.,
\end{eqnarray}
and we set $ f^-(0)=f_0$.
\end{remark}

The following result was shown by Theorem 24 in \cite{D-V-1}.
\begin{thm}\label{thm-4}
(\cite{D-V-1}, Existence and Uniqueness) Let $\eta\in L^{\infty}(\mathbb{T}^N)$. Assume Hypothesis H holds, then there is a unique kinetic solution $u$ to equation (\ref{P-19}) with initial datum $\eta$.
\end{thm}
Moreover, by Corollary 16 in \cite{D-V-1}, it follows that
\begin{cor}\label{cor-1}
(Continuity in time). Let $\eta\in L^{\infty}(\mathbb{T}^N)$. Assume Hypothesis H is in force, then for every $p\in [1, +\infty)$, the kinetic solution $u$ to (\ref{P-19}) with initial datum $\eta$ has almost sure continuous trajectories in $L^p(\mathbb{T}^N)$.
\end{cor}



\section{Small time asymptotics and statement of our main result}
In the rest part, we take $T=1$. Let $0<\varepsilon\leq 1$, by the scaling property of the Brownian motion, it is readily to deduce that $u(\varepsilon t)$ coincides in law with the solution of the following equation:
\begin{eqnarray}\label{eqq-5}
u^{\varepsilon}_{\eta}(t,x)+\varepsilon\int^t_0 div (A(u^{\varepsilon}_{\eta}(s)))ds=\eta(x)
+\sqrt{\varepsilon} \int^t_0 \sum_{k\geq 1}g_k(x,u^{\varepsilon}_{\eta}(s,x)) d\beta_k(s).
\end{eqnarray}
 By Theorem \ref{thm-4}, there is a unique kinetic solution $u^{\varepsilon}_{\eta}$. Applying Sections 6 and 7 in \cite{DHV} with $A=0$, we obtain for any $p\geq 1$,
\begin{eqnarray}\label{qq-27}
\sup_{0<\varepsilon\leq 1}\mathbb{E} \underset{0\leq s\leq 1}{{\rm{ess\sup}}}\ \|u^{\varepsilon}_{\eta}(s)\|^p_{L^p(\mathbb{T}^N)}<\infty.
\end{eqnarray}
By Lemma \ref{lem-1}, there exists a kinetic measure $m^{\varepsilon}_1$ such that  $f_1(x,t,\xi):=I_{u^{\varepsilon}_{\eta}(x,t)>\xi}$ satisfies that
 for all $t\in [0,1)$ and $\varphi\in C^1_c(\mathbb{T}^N\times \mathbb{R})$,
\begin{eqnarray}\notag
\langle f^+_1(t),\varphi\rangle&=&\langle f_{1,0}, \varphi\rangle+\varepsilon\int^t_0\langle f_1(s), a(\xi)\cdot \nabla \varphi\rangle ds\\
\notag
&&+\sqrt{\varepsilon}\sum_{k\geq 1}\int^t_0\int_{\mathbb{T}^N}\int_{\mathbb{R}}g_k(x,\xi)\varphi(x,\xi)d\nu^{1,\varepsilon}_{x,s}(\xi)dxd\beta_k(s)\\
\label{eq-11}
&& +\frac{\varepsilon}{2}\int^t_0\int_{\mathbb{T}^N}\int_{\mathbb{R}}\partial_{\xi}\varphi(x,\xi)G^2(x,\xi)d\nu^{1,\varepsilon}_{x,s}(\xi)dxds- \langle m^{\varepsilon}_1,\partial_{\xi} \varphi\rangle([0,t]), \quad a.s.,
\end{eqnarray}
where $\nu^{1,\varepsilon}_{x,s}(\xi)=-\partial_{\xi}f_1(x,s,\xi)=\delta_{u^{\varepsilon}_{\eta}(x,s)=\xi}$ and we set $f^+_1(1)=f_1(1)$.

For $h\in L^2([0,1];U)$ with the form $h(t)=\sum_{k\geq 1}h_k(t)e_k$, consider the following deterministic equation:
\begin{eqnarray*}
\left\{
  \begin{array}{ll}
    du^h(t,x)=\sum_{k\geq 1}g_k(x,u^h(t,x))h_k(t)dt, &  \\
    u^h(0)=\eta. &
  \end{array}
\right.
\end{eqnarray*}
{\color{rr}Applying Theorem 5.1 and Theorem 5.3 in \cite{DWZZ} with $A=0$, there exists a unique kinetic solution $u^h_{\eta}$ in the space $L^1([0,1];L^1(\mathbb{T}^N))$.}
Define
\[
R(h)=\frac{1}{2}\sum_{k\geq 1}\int^1_0|h_k(t)|^2dt.
\]
For $\varrho\in L^1([0,1];L^1(\mathbb{T}^N))$, define
\[
\mathcal{L}_{\varrho}=\Big\{h\in L^2([0,1];U): \varrho(\cdot)=u^h_{\eta}(\cdot)\Big\}
\]
Set
\begin{eqnarray}\label{eqq-6}
I(\varrho)=\left\{
             \begin{array}{ll}
              \inf_{h\in \mathcal{L}_{\varrho}}R(h) , & {\rm{if}}\  \mathcal{L}_{\varrho}\neq \emptyset \\
              +\infty , &  {\rm{if}}\  \mathcal{L}_{\varrho}= \emptyset
             \end{array}
           \right.
\end{eqnarray}

{\color{rr}For any initial value $\eta\in L^{\infty}(\mathbb{T}^N)$, let $u^{\varepsilon}_{\eta}$ be the kinetic solution of  (\ref{eqq-5}). Denote by $\mu^{\varepsilon}_{\eta}$ the law of $u^{\varepsilon}_{\eta}$ on the space $L^1([0,1];L^1(\mathbb{T}^N))$.} The main result of this article reads as follows.
\begin{thm}\label{thm-1}
 Let the initial value $\eta\in L^{\infty}(\mathbb{T}^N)$. Assume Hypotheses H is in force, then $\{\mu^{\varepsilon}_{\eta},\varepsilon>0\}$ satisfies a large deviation principle with the rate function $I(\cdot)$ defined by (\ref{eqq-6}), that is,
\begin{description}
  \item[(i)] For any closed subset $F\subset L^1([0,1];L^1(\mathbb{T}^N)) $,
  \[
  \limsup_{\varepsilon\rightarrow 0}\varepsilon \log \mu^{\varepsilon}_{\eta}(F)\leq -\inf_{\varrho\in F}I(\varrho).
  \]
  \item[(ii)] For any open subset $G\subset L^1([0,1];L^1(\mathbb{T}^N)) $,
  \[
  \liminf_{\varepsilon\rightarrow 0}\varepsilon \log \mu^{\varepsilon}_{\eta}(G)\geq -\inf_{\varrho\in G}I(\varrho).
  \]
\end{description}
\end{thm}
\begin{proof} Applying Theorem 24 in \cite{D-V-1} with $A=0$, we know that there exists a unique kinetic solution $v^{\varepsilon}_{\eta}$ to the following stochastic equation
\begin{eqnarray}\label{eq-7}
v^{\varepsilon}_{\eta}(t,x)=\eta(x)+\sqrt{\varepsilon}\int^t_0\sum_{k\geq 1}g_k(x, v^{\varepsilon}_{\eta}(s,x))d\beta_k(s).
\end{eqnarray}
Let $\vartheta^{\varepsilon}_{\eta}$ be the law of $v^{\varepsilon}_{\eta}(\cdot)$ on $L^1([0,1];L^1(\mathbb{T}^N))$. According to {\color{rr}Theorem 4.2} in \cite{DWZZ} with $A=0$,
it follows that $\vartheta^{\varepsilon}_{\eta}$ satisfies a large deviation principle with the rate function $I(\cdot)$. Based on Theorem 4.2.13 in \cite{DZ}, it suffices to show that two families of the probability measures $\mu^{\varepsilon}_{\eta}$ and $\vartheta^{\varepsilon}_{\eta}$ are exponentially equivalent, that is, for any $\iota >0$,
\begin{eqnarray}\label{eq-8}
\lim_{\varepsilon\rightarrow 0}\varepsilon \log \mathbb{P}\Big(\|u^{\varepsilon}_{\eta}-v^{\varepsilon}_{\eta}\|_{L^1([0,1];L^1(\mathbb{T}^N))}>\iota\Big)=-\infty.
\end{eqnarray}
\end{proof}

\textbf{From now on, for the sake of simplicity, we denote by $u^{\varepsilon}=u^{\varepsilon}_{\eta}$ and $v^{\varepsilon}=v^{\varepsilon}_{\eta}$ when the initial value is not emphasized.}

\section{Proof of the main result}
Recall that $v^{\varepsilon}$ is the unique kinetic solution to (\ref{eq-7}). Applying Sections 6 and 7 in \cite{DHV} with $A=0, B=0$, we obtain that, for any $p\geq 1$,
\begin{eqnarray}\label{qq-28}
\sup_{0<\varepsilon\leq 1}\mathbb{E}\underset{0\leq s\leq 1}{{\rm{ess\sup}}}\ \|v^{\varepsilon}(s)\|^p_{L^p(\mathbb{T}^N)}<\infty.
\end{eqnarray}
Moreover, by Lemma \ref{lem-1}, there exists a kinetic measure $m^{\varepsilon}_2$ such that $f_2(x,t,\xi):=I_{v^{\varepsilon}(x,t)>\xi}$ satisfies that
for all $t\in [0,1)$ and $\varphi\in C^1_c(\mathbb{T}^N\times \mathbb{R})$,
\begin{eqnarray}\notag
\langle f^+_2(t),\varphi\rangle&=&\langle f_{2,0}, \varphi\rangle+\sqrt{\varepsilon}\sum_{k\geq 1}\int^t_0\int_{\mathbb{T}^N}\int_{\mathbb{R}}g_k(x,\xi)\varphi(x,\xi)d\nu^{2,\varepsilon}_{x,s}(\xi)dxd\beta_k(s)\\
\label{eq-15}
&& +\frac{\varepsilon}{2}\int^t_0\int_{\mathbb{T}^N}\int_{\mathbb{R}}\partial_{\xi}\varphi(x,\xi)G^2(x,\xi)d\nu^{2,\varepsilon}_{x,s}(\xi)dxds- \langle m^{\varepsilon}_2,\partial_{\xi} \varphi\rangle([0,t]), \quad  a.s.,
\end{eqnarray}
where $f_{2,0}=I_{\eta>\xi}$, $\nu^{2,\varepsilon}_{x,s}=-\partial_{\xi} f_2=\delta_{v^{\varepsilon}(x,s)=\xi}$ and we set $ f^+_2(1)=f_2(1)$.

Following the idea of Proposition 13 in \cite{D-V-1} and by utilizing the doubling variables method, we have the following result relating $u^{\varepsilon}$ and $v^{\varepsilon}$.
\begin{prp}\label{prp-1}
Assume Hypothesis H is in place. Let $u^{\varepsilon}$ and $v^{\varepsilon}$ be the kinetic solution of (\ref{eqq-5}) and (\ref{eq-7}), respectively. Then, for all $0< t< 1$, and non-negative test functions $\rho\in C^{\infty}(\mathbb{T}^N), \psi\in C^{\infty}_c(\mathbb{R})$, the corresponding functions $f_1(x,t,\xi):=I_{u^{\varepsilon}(x,t)>\xi}$ and $f_2(y,t,\zeta):=I_{v^{\varepsilon}(y,t)>\zeta}$ satisfy
\begin{eqnarray}\notag
&&\int_{(\mathbb{T}^N)^2}\int_{\mathbb{R}^2}\rho (x-y)\psi(\xi-\zeta)(f^{\pm}_1(x,s,\xi)\bar{f}^{\pm}_2(y,s,\zeta)+\bar{f}^{\pm}_1(x,s,\xi)f^{\pm}_2(y,s,\zeta))d\xi d\zeta dxdy\\
\notag
&\leq & \int_{(\mathbb{T}^N)^2}\int_{\mathbb{R}^2}\rho (x-y)\psi(\xi-\zeta)(f_{1,0}(x,\xi)\bar{f}_{2,0}(y,\zeta)+\bar{f}_{1,0}(x,\xi)f_{2,0}(y,\zeta))d\xi d\zeta dxdy\\
\label{eq-14-1}
&& +I(t)+J(t)+K(t), \quad a.s.,
\end{eqnarray}
where
\begin{eqnarray*}
I(t)&=&\varepsilon\int^t_0\int_{(\mathbb{T}^N)^2}\int_{\mathbb{R}^2}(f_1\bar{f}_2+\bar{f}_1f_2)(a (\xi)\cdot\nabla_x)\alpha d\xi d\zeta dxdyds\\
J(t)&=&\varepsilon\int^t_0\int_{(\mathbb{T}^N)^2}\int_{\mathbb{R}^2}\alpha \sum_{k\geq 1}|g_k(x,\xi)-g_k(y,\zeta)|^2d\nu^{1,\varepsilon}_{x,s}\otimes \nu^{2,\varepsilon}_{y,s}(\xi,\zeta)dxdyds\\
K(t)&=& 2\sqrt{\varepsilon}\sum_{k\geq 1}\int^t_0\int_{(\mathbb{T}^N)^2}\int_{\mathbb{R}^2}(g_k(x,\xi)-g_k(y,\zeta))\rho(x-y)\chi_1(\xi,\zeta)d\nu^{1,\varepsilon}_{x,s}\otimes \nu^{2,\varepsilon}_{y,s}(\xi,\zeta)dxdyd\beta_k(s),
\end{eqnarray*}
with $f_{1,0}(x,\xi)=I_{\eta(x)>\xi}, f_{2,0}(y,\zeta)=I_{\eta(y)>\zeta}$, $\alpha=\rho (x-y)\psi(\xi-\zeta)$, $\nu^{1,\varepsilon}_{x,s}=-\partial_{\xi}f_1(s,x,\xi)=\delta_{u^{\varepsilon}(x,s)=\xi}, \nu^{2,\varepsilon}_{y,s}=\partial_{\zeta}\bar{f}_2(s,y,\zeta)=\delta_{v^{\varepsilon}(y,s)=\zeta}$ and $\chi_1(\xi,\zeta)=\int^{\xi}_{-\infty}\psi(\xi'-\zeta)d\xi'=\int^{\xi-\zeta}_{-\infty}\psi(y)dy$.
\end{prp}

\begin{proof}
Let $\varphi_1\in C^{1}_c(\mathbb{T}^N_x\times \mathbb{R}_{\xi})$ and
$\varphi_2\in C^{1}_c(\mathbb{T}^N_y\times \mathbb{R}_{\zeta})$.
For all $t\in (0,1)$, according to (\ref{eq-11}), it yields
\begin{eqnarray*}\notag
\langle f^+_1(t),\varphi_1\rangle=\langle m^*_1, \partial_{\xi}\varphi_1\rangle([0,t])+F_1(t), \quad  a.s.,
\end{eqnarray*}
with
 \begin{eqnarray*}
 \langle m^*_1, \partial_{\xi}\varphi_1\rangle([0,t])&=&\langle f_{1,0}, \varphi_1\rangle\delta_0([0,t])+\varepsilon\int^t_0\langle f_1(s), a(\xi)\cdot \nabla_x \varphi_1(s)\rangle ds\\
 && +\frac{\varepsilon}{2}\int^t_0\int_{\mathbb{T}^N}\int_{\mathbb{R}}\partial_{\xi}\varphi_1(x,\xi)G^2(x,\xi)d\nu^{1,\varepsilon}_{x,s}(\xi)dxds- \langle m^{\varepsilon}_1,\partial_{\xi} \varphi_1\rangle([0,t]),
 \end{eqnarray*}
 and
  \begin{eqnarray*}
 F_1(t)=\sqrt{\varepsilon}\sum_{k\geq 1}\int^t_0\int_{\mathbb{T}^N}\int_{\mathbb{R}}g_k(x,\xi)\varphi_1(x,\xi)d\nu^{1,\varepsilon}_{x,s}(\xi)dxd\beta_k(s).
 \end{eqnarray*}
Similarly, by utilizing (\ref{eq-15}), we have
\begin{eqnarray*}\notag
\langle \bar{f}^+_2(t),\varphi_2\rangle=\langle \bar{m}^*_2, \partial_{\zeta}\varphi_2\rangle([0,t])+\bar{F}_2(t),
\end{eqnarray*}
where
\begin{eqnarray*}\notag
\langle \bar{m}^*_2, \partial_{\zeta}\varphi_2\rangle([0,t])&=&\langle \bar{f}_{2,0}, \varphi_2\rangle\delta_0([0,t]) -\frac{\varepsilon}{2}\int^t_0\int_{\mathbb{T}^N}\int_{\mathbb{R}}\partial_{\zeta}\varphi_2(y,\zeta)G^2(y,\zeta)d\nu^{2,\varepsilon}_{y,s}(\zeta)dyds\\
&& +\langle m^{\varepsilon}_2,\partial_{\zeta} \varphi_2\rangle([0,t]),
\end{eqnarray*}
and
\begin{eqnarray*}
\bar{F}_2(t)=-\sqrt{\varepsilon}\sum_{k\geq 1}\int^t_0\int_{\mathbb{T}^N}\int_{\mathbb{R}}g_k(y,\zeta)\varphi_2(y,\zeta)d\nu^{2,\varepsilon}_{y,s}(\zeta)dyd\beta_k(s).
\end{eqnarray*}
Clearly, $F_1(t)$ and $\bar{F}_2(t)$ are continuous martingales, $t\mapsto\langle m^*_1, \partial_{\xi}\varphi_1\rangle([0,t])$ and  $t\mapsto\langle \bar{m}^*_2, \partial_{\zeta}\varphi_2\rangle([0,t])$ are functions of finite variation. Moreover, it is shown in Remark 12 of \cite{D-V-1} that $\langle m^*_1, \partial_{\xi}\varphi_1\rangle(\{0\})=\langle f_{1,0}, \varphi_1\rangle$ and $\langle \bar{m}^*_2, \partial_{\zeta}\varphi_2\rangle(\{0\})=\langle \bar{f}_{2,0}, \varphi_2\rangle$.

Denote the duality distribution over $\mathbb{T}^N_x\times \mathbb{R}_\xi\times \mathbb{T}^N_y\times \mathbb{R}_\zeta$ by $\langle\langle\cdot,\cdot\rangle\rangle$. Let $\alpha(x,\xi,y,\zeta)=\varphi_1(x,\xi)\varphi_2(y,\zeta)$.
Applying It\^{o} formula to $F_1(t)\bar{F}_2(t)$, it yields
\begin{eqnarray*}
  F_1(t)\bar{F}_2(t)=\int^t_0F_1(s)d\bar{F}_2(s)+\int^t_0\bar{F}_2(s)dF_1(s)+[F_1,\bar{F}_2]_t,
\end{eqnarray*}
where $[F_1,\bar{F}_2]_t$ is the quadratic variation of $F_1$ and $\bar{F}_2$ at time $t$. Moreover,
according to Proposition (4.5) on P6 in \cite{RY99} and by using integration by parts for $\langle m^*_1, \partial_{\xi}\varphi_1\rangle([0,t])\langle \bar{m}^*_2, \partial_{\zeta}\varphi_2\rangle([0,t])$, it yields that
\begin{eqnarray*}
\langle m^*_1, \partial_{\xi}\varphi_1\rangle([0,t])\langle \bar{m}^*_2, \partial_{\zeta}\varphi_2\rangle([0,t])&=&\langle m^*_1, \partial_{\xi}\varphi_1\rangle(\{0\})\langle \bar{m}^*_2, \partial_{\zeta}\varphi_2\rangle(\{0\})\\
&& +\int_{(0,t]}\langle m^*_1, \partial_{\xi}\varphi_1\rangle([0,s))d\langle \bar{m}^*_2, \partial_{\zeta}\varphi_2\rangle(s)\\
&& +\int_{(0,t]}\langle \bar{m}^*_2, \partial_{\zeta}\varphi_2\rangle([0,s])d\langle m^*_1, \partial_{\xi}\varphi_1\rangle(s).
\end{eqnarray*}
Since $\bar{F}_2$ is continuous, we have
\begin{eqnarray*}
\langle m^*_1, \partial_{\xi}\varphi_1\rangle([0,t])\bar{F}_2(t)=\int^t_0\langle m^*_1, \partial_{\xi}\varphi_1\rangle([0,s])d\bar{F}_2(s)+\int^t_0\bar{F}_2(s)\langle m^*_1, \partial_{\xi}\varphi_1\rangle(ds),
\end{eqnarray*}
and we have the similar formula for $\langle \bar{m}^*_2, \partial_{\zeta}\varphi_2\rangle([0,t])F_2(t)$.

Based on the above formulas and using (\ref{e-80}), we obtain that
\begin{eqnarray*}
\langle f^+_1(t), \varphi_1\rangle\langle \bar{f}^+_2(t),\varphi_2 \rangle=\langle\langle f^+_1(t)\bar{f}^+_2(t), \alpha \rangle\rangle
\end{eqnarray*}
satisfies
\begin{eqnarray}\notag
\langle\langle f^+_1(t)\bar{f}^+_2(t), \alpha \rangle\rangle
&=& \langle\langle f_{1,0}\bar{f}_{2,0}, \alpha \rangle\rangle
+\varepsilon\int^t_0\int_{(\mathbb{T}^N)^2}\int_{\mathbb{R}^2}f_1\bar{f}_2(a(\xi)\cdot \nabla_x) \alpha d\xi d\zeta dxdyds\\ \notag
 && -\frac{\varepsilon}{2}\int^t_0\int_{(\mathbb{T}^N)^2}\int_{\mathbb{R}^2}f_1(s,x,\xi)\partial_{\zeta}\alpha G^2(y,\zeta)d\xi d\nu^{2,\varepsilon}_{y,s}(\zeta)dxdyds\\ \notag
 && +\frac{\varepsilon}{2}\int^t_0\int_{(\mathbb{T}^N)^2}\int_{\mathbb{R}^2}\bar{f}_2(s,y,\zeta)\partial_{\xi}\alpha G^2(x,\xi)d\zeta d\nu^{1,\varepsilon}_{x,s}(\xi)dxdyds\\ \notag
  && -\varepsilon \int^t_0\int_{(\mathbb{T}^N)^2}\int_{\mathbb{R}^2}G_{1,2}(x,y,\xi,\zeta)\alpha d\nu^{1,\varepsilon}_{x,s}\otimes \nu^{2,\varepsilon}_{y,s}(\xi,\zeta)dxdyds\\ \notag
 && +\int_{(0,t]}\int_{(\mathbb{T}^N)^2}\int_{\mathbb{R}^2}f^{-}_1(s,x,\xi)\partial_{\zeta} \alpha dm^{\varepsilon}_2(y,\zeta,s)d\xi dx\\ \notag
 && -\int_{(0,t]}\int_{(\mathbb{T}^N)^2}\int_{\mathbb{R}^2}\bar{f}^+_2(s,y,\zeta)\partial_{\xi} \alpha dm^{\varepsilon}_1(x,\xi,s)d\zeta dy\\ \notag
 && -\sqrt{\varepsilon}\sum_{k\geq 1}\int^t_0\int_{(\mathbb{T}^N)^2}\int_{\mathbb{R}^2}f_1(s,x,\xi)g_k(y,\zeta)\alpha d\nu^{2,\varepsilon}_{y,s}(\zeta)dydxd\xi d\beta_k(s)\\ \notag
 && +\sqrt{\varepsilon}\sum_{k\geq 1}\int^t_0\int_{(\mathbb{T}^N)^2}\int_{\mathbb{R}^2}\bar{f}_2(s,y,\zeta)g_k(x,\xi)\alpha d\nu^{1,\varepsilon}_{x,s}(\xi)dydxd\zeta d\beta_k(s)\\
\label{P-10}
&=:& \langle\langle f_{1,0}\bar{f}_{2,0}, \alpha \rangle\rangle+\sum^8_{i=1}I_i(t), \quad a.s.,
\end{eqnarray}
where $G^2(x,\xi)=\sum_{k\geq 1}|g_k(x,\xi)|^2$ and $G_{1,2}(x,\xi,y,\zeta)=\sum_{k\geq 1}g_k(x,\xi)g_k(y,\zeta)$.

Similarly, we  have
\begin{eqnarray}\notag
\langle\langle \bar{f}^+_1(t)f^+_2(t), \alpha \rangle\rangle
&=& \langle\langle \bar{f}_{1,0}f_{2,0}, \alpha \rangle\rangle
+\varepsilon\int^t_0\int_{(\mathbb{T}^N)^2}\int_{\mathbb{R}^2}\bar{f}_1{f}_2(a(\xi)\cdot \nabla_x) \alpha d\xi d\zeta dxdyds\\ \notag
 && +\frac{\varepsilon}{2}\int^t_0\int_{(\mathbb{T}^N)^2}\int_{\mathbb{R}^2}\bar{f}_1(s,x,\xi)\partial_{\zeta}\alpha G^2(y,\zeta)d\xi d\nu^{2,\varepsilon}_{y,s}(\zeta)dxdyds\\ \notag
 && -\frac{\varepsilon}{2}\int^t_0\int_{(\mathbb{T}^N)^2}\int_{\mathbb{R}^2}{f}_2(s,y,\zeta)\partial_{\xi}\alpha G^2(x,\xi)d\zeta d\nu^{1,\varepsilon}_{x,s}(\xi)dxdyds\\ \notag
  && -\varepsilon \int^t_0\int_{(\mathbb{T}^N)^2}\int_{\mathbb{R}^2}G_{1,2}(x,y,\xi,\zeta)\alpha d\nu^{1,\varepsilon}_{x,s}\otimes \nu^{2,\varepsilon}_{y,s}(\xi,\zeta)dxdyds\\ \notag
 && -\int_{(0,t]}\int_{(\mathbb{T}^N)^2}\int_{\mathbb{R}^2}\bar{f}^+_1(s,x,\xi)\partial_{\zeta} \alpha dm^{\varepsilon}_2(y,\zeta,s)d\xi dx\\ \notag
 && +\int_{(0,t]}\int_{(\mathbb{T}^N)^2}\int_{\mathbb{R}^2}{f}^{-}_2(s,y,\zeta)\partial_{\xi} \alpha dm^{\varepsilon}_1(x,\xi,s)d\zeta dy\\ \notag
 && +\sqrt{\varepsilon}\sum_{k\geq 1}\int^t_0\int_{(\mathbb{T}^N)^2}\int_{\mathbb{R}^2}\bar{f}_1(s,x,\xi)g_k(y,\zeta)\alpha d\nu^{2,\varepsilon}_{y,s}(\zeta)dydxd\xi d\beta_k(s)\\ \notag
 && -\sqrt{\varepsilon}\sum_{k\geq 1}\int^t_0\int_{(\mathbb{T}^N)^2}\int_{\mathbb{R}^2}{f}_2(s,y,\zeta)g_k(x,\xi)\alpha d\nu^{1,\varepsilon}_{x,s}(\xi)dydxd\zeta d\beta_k(s)\\
\label{P-10-1}
&=:& \langle\langle \bar{f}_{1,0}{f}_{2,0}, \alpha \rangle\rangle+\sum^8_{i=1}\bar{I}_i(t), \quad a.s..
\end{eqnarray}
Noting that $C^{1}_c(\mathbb{T}^N_x\times \mathbb{R}_\xi)\otimes C^{1}_c( \mathbb{T}^N_y\times \mathbb{R}_\zeta)$ is dense in $C^{1}_c(\mathbb{T}^N_x\times \mathbb{R}_\xi\times \mathbb{T}^N_y\times \mathbb{R}_\zeta)$ and the assumption that $\alpha$ is compactly supported can be relaxed thanks to (\ref{equ-37}), (\ref{qq-27}) and (\ref{qq-28}). By truncation, we can take $\alpha\in C^{\infty}_b(\mathbb{T}^N_x\times \mathbb{R}_\xi\times \mathbb{T}^N_y\times \mathbb{R}_\zeta)$ compactly supported in a neighbourhood of the diagonal
\[
\Big\{(x,\xi,x,\xi); x\in \mathbb{T}^N, \xi\in \mathbb{R}\Big\},
\]
with the form $\alpha=\rho(x-y)\psi(\xi-\zeta)$, which implies the following remarkable identities
\begin{eqnarray}\label{P-11}
(\nabla_x+\nabla_y)\alpha=0, \quad (\partial_{\xi}+\partial_{\zeta})\alpha=0.
\end{eqnarray}

From now on, we devote to making estimates of $I_i, \bar{I}_i,$ for $i=1,\cdot\cdot\cdot, 8$. Clearly, it holds that
\begin{eqnarray*}
I_1(t)+\bar{I}_1(t)&=&\varepsilon\int^t_0\int_{(\mathbb{T}^N)^2}\int_{\mathbb{R}^2}(f_1\bar{f}_2+\bar{f}_1{f}_2)(a(\xi)\cdot \nabla_x) \alpha d\xi d\zeta dxdyds\\
&=:&I(t).
\end{eqnarray*}
{\color{rr}In view of (\ref{P-11}), it holds that
\begin{eqnarray*}
I_5&=&-\int_{(0,t]}\int_{(\mathbb{T}^N)^2}\int_{\mathbb{R}^2}f^-_1(s,x,\xi)\partial_{\xi} \alpha dm^{\varepsilon}_2(y,\zeta,s)d\xi dx\\
&=&-\int_{(0,t]}\int_{(\mathbb{T}^N)^2}\int_{\mathbb{R}^2} \alpha dm^{\varepsilon}_2(y,\zeta,s)d\nu^{1,\varepsilon,-}_{x,s}(\xi)\leq 0, \quad a.s.,
\end{eqnarray*}
and
\begin{eqnarray*}
I_6&=&\int_{(0,t]}\int_{(\mathbb{T}^N)^2}\int_{\mathbb{R}^2}\bar{f}^+_2(s,y,\zeta)\partial_{\zeta} \alpha dm^{\varepsilon}_1(x,\xi,s)d\zeta dy\\
&=&-\int_{(0,t]}\int_{(\mathbb{T}^N)^2}\int_{\mathbb{R}^2} \alpha dm^{\varepsilon}_1(x,\xi,s)d\nu^{2,\varepsilon,+}_{y,s}(\zeta)\leq 0, \quad a.s..
\end{eqnarray*}
By the same method as above, we deduce that $\bar{I}_5+ \bar{I}_6\leq 0$, a.s..
}

Moreover,
it is readily to deduce that
\begin{eqnarray*}
I_2+I_3+I_4&=&\bar{I}_2+\bar{I}_3+\bar{I}_4\\
&=&\frac{\varepsilon}{2}\int^t_0\int_{(\mathbb{T}^N)^2}\int_{\mathbb{R}^2}\alpha (G^2(x,\xi)+G^2(y,\zeta)-2G_{1,2}(x,y,\xi,\zeta))d\nu^{1,\varepsilon}_{x,s}\otimes \nu^{2,\varepsilon}_{y,s}(\xi,\zeta)dxdyds\\
&=& \frac{\varepsilon}{2}\int^t_0\int_{(\mathbb{T}^N)^2}\int_{\mathbb{R}^2}\alpha \sum_{k\geq 1}|g_k(x,\xi)-g_k(y,\zeta)|^2d\nu^{1,\varepsilon}_{x,s}\otimes \nu^{2,\varepsilon}_{y,s}(\xi,\zeta)dxdyds,
\end{eqnarray*}
hence,
\[
\sum^{4}_{i=2}(I_i+\bar{I}_i)=\varepsilon\int^t_0\int_{(\mathbb{T}^N)^2}\int_{\mathbb{R}^2}\alpha \sum_{k\geq 1}|g_k(x,\xi)-g_k(y,\zeta)|^2d\nu^{1,\varepsilon}_{x,s}\otimes \nu^{2,\varepsilon}_{y,s}(\xi,\zeta)dxdyds=:J(t).
\]
Define  $\chi_1(\xi,\zeta)=\int^{\xi}_{-\infty}\psi(\xi'-\zeta)d\xi'$, then
{\color{rr}
\begin{eqnarray*}
I_7(t)&=&-\sqrt{\varepsilon}\sum_{k\geq 1}\int^t_0\int_{(\mathbb{T}^N)^2}\int_{\mathbb{R}^2}f_1(s,x,\xi)g_k(y,\zeta)\rho(x-y)\partial_{\xi}\chi_1(\xi,\zeta) d\nu^{2,\varepsilon}_{y,s}(\zeta)dydxd\xi d\beta_k(s)\\
&=&-\sqrt{\varepsilon}\sum_{k\geq 1}\int^t_0\int_{(\mathbb{T}^N)^2}\int_{\mathbb{R}^2}g_k(y,\zeta)\rho(x-y)\chi_1(\xi,\zeta)d\nu^{1,\varepsilon}_{x,s}\otimes \nu^{2,\varepsilon}_{y,s}(\xi,\zeta)dxdyd\beta_k(s).
\end{eqnarray*}
}
Define $\chi_2(\xi,\zeta)=\int^{+\infty}_{\zeta}\psi(\xi-\zeta')d\zeta'$, then
{\color{rr}
\begin{eqnarray*}
I_8(t)&=&-\sqrt{\varepsilon}\sum_{k\geq 1}\int^t_0\int_{(\mathbb{T}^N)^2}\int_{\mathbb{R}^2}\bar{f}_2(s,y,\zeta)g_k(x,\xi)\rho(x-y)\partial_{\zeta}\chi_2(\xi,\zeta)d\nu^{1,\varepsilon}_{x,s}(\xi)dydxd\zeta d\beta_k(s)\\
&=&\sqrt{\varepsilon}\sum_{k\geq 1}\int^t_0\int_{(\mathbb{T}^N)^2}\int_{\mathbb{R}^2}g_k(x,\xi)\rho(x-y)\chi_2(\xi,\zeta)d\nu^{1,\varepsilon}_{x,s}\otimes \nu^{2,\varepsilon}_{y,s}(\xi,\zeta)dxdyd\beta_k(s).
\end{eqnarray*}
}
Since $\chi_1(\xi,\zeta)=\chi_2(\xi,\zeta)=\int^{\xi-\zeta}_{-\infty}\psi(y)dy$, we get
\begin{eqnarray*}
I_7(t)+I_8(t)
=\sqrt{\varepsilon}\sum_{k\geq 1}\int^t_0\int_{(\mathbb{T}^N)^2}\int_{\mathbb{R}^2}(g_k(x,\xi)-g_k(y,\zeta))\rho(x-y)\chi_1(\xi,\zeta)d\nu^{1,\varepsilon}_{x,s}\otimes \nu^{2,\varepsilon}_{y,s}(\xi,\zeta)dxdyd\beta_k(s).
\end{eqnarray*}
Similarly, we deduce that
\begin{eqnarray*}
\bar{I}_7(t)+\bar{I}_8(t)
=\sqrt{\varepsilon}\sum_{k\geq 1}\int^t_0\int_{(\mathbb{T}^N)^2}\int_{\mathbb{R}^2}(g_k(x,\xi)-g_k(y,\zeta))\rho(x-y)\chi_1(\xi,\zeta)d\nu^{1,\varepsilon}_{x,s}\otimes \nu^{2,\varepsilon}_{y,s}(\xi,\zeta)dxdyd\beta_k(s).
\end{eqnarray*}
Thus, it yields
\begin{eqnarray*}
\sum^8_{i=7}(I_i+\bar{I}_i)&=&2\sqrt{\varepsilon}\sum_{k\geq 1}\int^t_0\int_{(\mathbb{T}^N)^2}\int_{\mathbb{R}^2}(g_k(x,\xi)-g_k(y,\zeta))\rho(x-y)\chi_1(\xi,\zeta)d\nu^{1,\varepsilon}_{x,s}\otimes \nu^{2,\varepsilon}_{y,s}(\xi,\zeta)dxdyd\beta_k(s)\\
&=:&K(t).
\end{eqnarray*}
Combining all the previous estimates, it follows that
\begin{eqnarray}\notag
&&\int_{(\mathbb{T}^N)^2}\int_{\mathbb{R}^2}\rho (x-y)\psi(\xi-\zeta)(f^+_1(x,t,\xi)\bar{f}^+_2(y,t,\zeta)+\bar{f}^+_1(x,t,\xi)f^+_2(y,t,\zeta))d\xi d\zeta dxdy\\
\notag
&\leq & \int_{(\mathbb{T}^N)^2}\int_{\mathbb{R}^2}\rho (x-y)\psi(\xi-\zeta)(f_{1,0}(x,\xi)\bar{f}_{2,0}(y,\zeta)+\bar{f}_{1,0}(x,\xi)f_{2,0}(y,\zeta))d\xi d\zeta dxdy\\
\label{eq-14}
&&\ +I(t)+J(t)+K(t), \quad a.s..
\end{eqnarray}
Taking $t_n\uparrow t$, we have (\ref{eq-14}) holds for $f^+_i(t_n)$ and let $n\rightarrow \infty$, we get (\ref{eq-14}) holds for $f^-_i(t)$. We complete the proof.

\end{proof}

Now, we are ready to proceed with the proof of (\ref{eq-8}), which implies the main result Theorem \ref{thm-1}.
\begin{prp}\label{prp-2}
 For any $\iota>0$, it holds that
\begin{eqnarray}
\lim_{\varepsilon\rightarrow 0}\varepsilon \log \mathbb{P}\left(\|u^{\varepsilon}-v^{\varepsilon}\|_{L^1([0,1];L^1(\mathbb{T}^N))}>\iota\right)=-\infty.
\end{eqnarray}
\end{prp}
\begin{proof}
Let $\rho_{\gamma}, \psi_{\delta}$ be approximations to the identity on $\mathbb{T}^N$ and $\mathbb{R}$, respectively. That is, let $\rho\in C^{\infty}(\mathbb{T}^N)$, $\psi\in C^{\infty}_c(\mathbb{R})$ be symmetric non-negative functions such as $\int_{\mathbb{T}^N}\rho =1$, $\int_{\mathbb{R}}\psi =1$ and supp$\psi \subset (-1,1)$. We define
\[
\rho_{\gamma}(x)=\frac{1}{\gamma^N}\rho\Big(\frac{x}{\gamma}\Big), \quad \psi_{\delta}(\xi)=\frac{1}{\delta}\psi\Big(\frac{\xi}{\delta}\Big).
\]
Letting $\rho:=\rho_{\gamma}(x-y)$ and $\psi:=\psi_{\delta}(\xi-\zeta)$ in Proposition \ref{prp-1}, we get from (\ref{eq-14-1}) that
\begin{eqnarray*}
&&\int_{(\mathbb{T}^N)^2}\int_{\mathbb{R}^2}\rho_{\gamma} (x-y)\psi_{\delta}(\xi-\zeta)(f^{\pm}_1(x,t,\xi)\bar{f}^{\pm}_2(y,t,\zeta)+\bar{f}^{\pm}_1(x,t,\xi)f^{\pm}_2(y,t,\zeta))d\xi d\zeta dxdy\\
&\leq & \int_{(\mathbb{T}^N)^2}\int_{\mathbb{R}^2}\rho_{\gamma} (x-y)\psi_{\delta}(\xi-\zeta)(f_{1,0}(x,\xi)\bar{f}_{2,0}(y,\zeta)+\bar{f}_{1,0}(x,\xi)f_{2,0}(y,\zeta))d\xi d\zeta dxdy\\
&&\  +\tilde{I}(t)+\tilde{J}(t)+\tilde{K}(t),\quad a.s.,
\end{eqnarray*}
where $\tilde{I}, \tilde{J}, \tilde{K}$ are the corresponding $I,J,K$ in the statement of Proposition \ref{prp-1} with $\rho$, $\psi$ replaced by $\rho_{\gamma}$, $\psi_{\delta}$, respectively. For simplicity, we still denote by $\chi_1(\xi,\zeta)$ with $\psi$ replaced by $\psi_{\delta}$.

For any $t\in [0,1]$, define the error term
\begin{eqnarray}\notag
&&\mathcal{E}_t(\gamma,\delta)\\ \notag
&:=&\int_{(\mathbb{T}^N)^2}\int_{\mathbb{R}^2}(f^{\pm}_1(x,t,\xi)\bar{f}^{\pm}_2(y,t,\zeta)+\bar{f}^{\pm}_1(x,t,\xi){f}^{\pm}_2(y,t,\zeta))\rho_{\gamma}(x-y)\psi_{\delta}(\xi-\zeta)dxdyd\xi d\zeta\\
\label{qq-3}
&&-\int_{\mathbb{T}^N}\int_{\mathbb{R}}(f^{\pm}_1(x,t,\xi)\bar{f}^{\pm}_2(x,t,\xi)+\bar{f}^{\pm}_1(x,t,\xi)f^{\pm}_2(x,t,\xi))d\xi dx.
\end{eqnarray}
By utilizing $\int_{\mathbb{R}}\psi_{\delta}(\xi-\zeta)d\zeta=1$, $\int^{\xi}_{\xi-\delta}\psi_{\delta}(\xi-\zeta)d\zeta=\frac{1}{2}$ and $\int_{(\mathbb{T}^N)^2}\rho_{\gamma}(x-y)dxdy\leq1$, we deduce that
\begin{eqnarray}\notag
&&\Big|\int_{(\mathbb{T}^N)^2}\int_{\mathbb{R}}\rho_{\gamma}(x-y)f^{\pm}_1(x,t,\xi)\bar{f}^{\pm}_2(y,t,\xi)d\xi dxdy\\ \notag
&&-\int_{(\mathbb{T}^N)^2}\int_{\mathbb{R}^2}f^{\pm}_1(x,t,\xi)\bar{f}^{\pm}_2(y,t,\zeta)\rho_{\gamma}(x-y)\psi_{\delta}(\xi-\zeta)dxdyd\xi d\zeta\Big|\\ \notag
&=&\Big|\int_{(\mathbb{T}^N)^2}\rho_{\gamma}(x-y)\int_{\mathbb{R}}I_{u^{\varepsilon, \pm}(x,t)>\xi}\int_{\mathbb{R}}\psi_{\delta}(\xi-\zeta)(I_{v^{\varepsilon,\pm}(y,t)\leq\xi}-I_{v^{\varepsilon,\pm}(y,t)\leq \zeta})d\zeta d\xi dxdy\Big|\\ \notag
&\leq&\int_{(\mathbb{T}^N)^2}\int_{\mathbb{R}}\rho_{\gamma}(x-y)I_{u^{\varepsilon,\pm}(x,t)>\xi}\int^{\xi}_{\xi-\delta}\psi_{\delta}(\xi-\zeta)I_{\zeta<v^{\varepsilon,\pm}(y,t)\leq\xi} d\zeta d\xi dxdy\\ \notag
&&\ +\int_{(\mathbb{T}^N)^2}\int_{\mathbb{R}}\rho_{\gamma}(x-y)I_{u^{\varepsilon,\pm}(x,t)>\xi}\int^{\xi+\delta}_{\xi}\psi_{\delta}(\xi-\zeta)I_{\xi<v^{\varepsilon,\pm}(y,t)\leq\zeta} d\zeta d\xi dxdy\\ \notag
&\leq& {\color{rr}\frac{1}{2}\int_{(\mathbb{T}^N)^2}\rho_{\gamma}(x-y)I_{{u^{\varepsilon,\pm}(x,t)>v^{\varepsilon,\pm}(y,t)}}\int^{min\{u^{\varepsilon,\pm}(x,t),v^{\varepsilon,\pm}(y,t)+\delta\}}_{v^{\varepsilon,\pm}(y,t)}d\xi dxdy}\\ \notag
&&\ {\color{rr}+\frac{1}{2}\int_{(\mathbb{T}^N)^2}\rho_{\gamma}(x-y)I_{{v^{\varepsilon,\pm}(y,t)-\delta<u^{\varepsilon,\pm}(x,t)}}\int^{min\{u^{\varepsilon,\pm}(x,t),v^{\varepsilon,\pm}(y,t)\}}_{v^{\varepsilon,\pm}(y,t)-\delta}d\xi dxdy}\\ \notag
&=& {\color{rr}\frac{\delta}{2}\int_{(\mathbb{T}^N)^2}\rho_{\gamma}(x-y) I_{{ u^{\varepsilon,\pm}(x,t)>v^{\varepsilon,\pm}(y,t)+\delta }}dxdy}\\ \notag
&&{\color{rr}+\frac{1}{2}\int_{(\mathbb{T}^N)^2}\rho_{\gamma}(x-y)I_{{v^{\varepsilon,\pm}(y,t)< u^{\varepsilon,\pm}(x,t)\leq v^{\varepsilon,\pm}(y,t)+\delta }}(u^{\varepsilon,\pm}(x,t)-v^{\varepsilon,\pm}(y,t) dxdy}\\ \notag
&&{\color{rr}+\frac{\delta}{2}\int_{(\mathbb{T}^N)^2}\rho_{\gamma}(x-y)I_{{v^{\varepsilon,\pm}(y,t)<u^{\varepsilon,\pm}(x,t)}}dxdy}\\ \notag
&&{\color{rr}+\frac{1}{2}\int_{(\mathbb{T}^N)^2}\rho_{\gamma}(x-y)I_{{v^{\varepsilon,\pm}(y,t)-\delta<u^{\varepsilon,\pm}(x,t)\leq v^{\varepsilon,\pm}(y,t)}}(u^{\varepsilon,\pm}(x,t)-v^{\varepsilon,\pm}(y,t)+\delta) dxdy}\\
\label{e-23}
&\leq & 2\delta, \quad a.s..
\end{eqnarray}
Similarly,
\begin{eqnarray}\notag
&&\Big|\int_{(\mathbb{T}^N)^2}\int_{\mathbb{R}}\rho_{\gamma}(x-y)\bar{f}^{\pm}_1(x,t,\xi){f}^{\pm}_2(y,t,\xi)d\xi dxdy\\
\label{e-22}
&&-\int_{(\mathbb{T}^N)^2}\int_{\mathbb{R}^2}\bar{f}^{\pm}_1(x,t,\xi){f}^{\pm}_2(y,t,\zeta)\rho_{\gamma}(x-y)\psi_{\delta}(\xi-\zeta)dxdyd\xi d\zeta\Big|
\leq  2\delta, \quad a.s..
\end{eqnarray}
Moreover, it follows that
\begin{eqnarray}\notag
&&\Big|\int_{(\mathbb{T}^N)^2}\int_{\mathbb{R}}\rho_{\gamma}(x-y)f^{\pm}_1(x,t,\xi)\bar{f}^{\pm}_2(y,t,\xi)d\xi dxdy-\int_{\mathbb{T}^N}\int_{\mathbb{R}}f^{\pm}_1(x,t,\xi)\bar{f}^{\pm}_2(x,t,\xi)d\xi dx\Big|\\ \notag
&=&\Big|\int_{\mathbb{T}^N}\int_{|z|<\gamma}\int_{\mathbb{R}}\rho_{\gamma}(z)f^{\pm}_1(x,t,\xi)\bar{f}^{\pm}_2(x-z,t,\xi)d\xi dxdz-\int_{\mathbb{T}^N}\int_{|z|<\gamma}\int_{\mathbb{R}}\rho_{\gamma}(z)f^{\pm}_1(x,t,\xi)\bar{f}^{\pm}_2(x,t,\xi)d\xi dxdz\Big|\\ \notag
&\leq&\sup_{|z|<\gamma}\int_{\mathbb{T}^N}\int_{\mathbb{R}}f^{\pm}_1(x,t,\xi)|\bar{f}^{\pm}_2(x-z,t,\xi)-\bar{f}^{\pm}_2(x,t,\xi)|d\xi dx\\ \notag
&\leq& \sup_{|z|<\gamma}\int_{\mathbb{T}^N}\int_{\mathbb{R}}|-f^{\pm}_2(x-z,t,\xi)+I_{0>\xi}-I_{0>\xi}+f^{\pm}_2(x,t,\xi)|d\xi dx\\
\label{e-43}
&=& \sup_{|z|<\gamma}\int_{\mathbb{T}^N}\int_{\mathbb{R}}|\Lambda_{f^{\pm}_2}(x-z,t,\xi)-\Lambda_{f^{\pm}_2}(x,t,\xi)|d\xi dx,   \quad a.s..
\end{eqnarray}
In view of the integrability of $\Lambda_{f^{\pm}_2}$, it yields that for a countable sequence $\gamma_n\downarrow 0$, (\ref{e-43}) holds a.s. for all $n$, hence, passing to the limit $n\rightarrow \infty$, we get
\begin{eqnarray}\label{qq-1}
\lim_{n\rightarrow \infty}\Big|\int_{(\mathbb{T}^N)^2}\int_{\mathbb{R}}\rho_{\gamma}(x-y)f^{\pm}_1(x,t,\xi)\bar{f}^{\pm}_2(y,t,\xi)d\xi dxdy-\int_{\mathbb{T}^N}\int_{\mathbb{R}}f^{\pm}_1(x,t,\xi)\bar{f}^{\pm}_2(x,t,\xi)d\xi dx\Big|= 0, \ a.s..
\end{eqnarray}
Similarly, it holds that
\begin{eqnarray}\label{qq-2}
\lim_{n\rightarrow \infty}\Big|\int_{(\mathbb{T}^N)^2}\int_{\mathbb{R}}\rho_{\gamma}(x-y)\bar{f}^{\pm}_1(x,t,\xi)f^{\pm}_2(y,t,\xi)d\xi dxdy-\int_{\mathbb{T}^N}\int_{\mathbb{R}}\bar{f}^{\pm}_1(x,t,\xi)f^{\pm}_2(x,t,\xi)d\xi dx\Big|=0, \ a.s..
\end{eqnarray}
By a similar argument, passing to the limit $\delta\rightarrow 0$, it follows from  (\ref{e-23})-(\ref{qq-2}) that
\begin{eqnarray}\notag
\lim_{n\rightarrow \infty}\mathcal{E}_t(\gamma_n,\delta_n)=0, \quad a.s..
\end{eqnarray}
Without confusion, from now on, we write
\begin{eqnarray}\label{qq-4}
\lim_{\gamma, \delta\rightarrow 0}\mathcal{E}_t(\gamma,\delta)=0, \quad a.s..
\end{eqnarray}
In particular, when $t=0$, it holds that
\begin{eqnarray}\label{qq-5}
\lim_{\gamma, \delta\rightarrow 0}\mathcal{E}_0(\gamma,\delta)=0.
\end{eqnarray}
In the following, we aim to make estimates of $\tilde{I}(t)$, $\tilde{J}(t)$ and $\tilde{K}(t)$.
We start with the estimation of $\tilde{I}(t)$. Note that
\begin{eqnarray*}
\tilde{I}(t)&=& \varepsilon\int^t_0\int_{(\mathbb{T}^N)^2}\int_{\mathbb{R}^2}f_1\bar{f}_2(a(\xi)\cdot \nabla_x) \alpha d\xi d\zeta dxdyds\\
&& +\varepsilon\int^t_0\int_{(\mathbb{T}^N)^2}\int_{\mathbb{R}^2}\bar{f}_1{f}_2(a(\xi)\cdot \nabla_x) \alpha d\xi d\zeta dxdyds\\
&=:& \tilde{I}_1(t)+\tilde{I}_2(t).
\end{eqnarray*}
By Hypothesis H, we know that $a(\cdot)$ is polynomial growth with degree $q_0$, then $|a(\xi)|\leq \mathcal{N}(q_0)(1+|\xi|^{q_0})$ with $\mathcal{N}(q_0)<\infty$. As a result, it yields
\begin{eqnarray*}
|\tilde{I}_1(t)|\leq\varepsilon \mathcal{N}(q_0)\int^t_0\int_{(\mathbb{T}^N)^2}\int_{\mathbb{R}^2}f_1\bar{f}_2(1+|\xi|^{q_0})\psi_{\delta}(\xi-\zeta)|\nabla_x \rho_{\gamma}(x-y)| d\xi d\zeta dxdyds, \quad a.s..
\end{eqnarray*}
Define
\[
\Gamma(\xi,\zeta)=\int^{\infty}_{\zeta}\int^{\xi}_{-\infty}(1+|\xi'|^{q_0})\psi_{\delta}(\xi'-\zeta')d\xi'd\zeta',
\]
then
\begin{eqnarray*}
|\tilde{I}_1(t)|\leq \varepsilon \mathcal{N}(q_0)\int^t_0\int_{(\mathbb{T}^N)^2}|\nabla_x \rho_{\gamma}(x-y)| \int_{\mathbb{R}^2}\Gamma(\xi,\zeta)d\nu^{1,\varepsilon}_{x,s}\otimes \nu^{2,\varepsilon}_{y,s}(\xi,\zeta) dxdyds, \quad a.s..
\end{eqnarray*}
Clearly, it yields
\begin{eqnarray*}
\Gamma(\xi,\zeta)&\leq&\int^{\infty}_{\zeta}\int_{|\xi''|<\delta, \xi''<\xi-\zeta'}(1+|\xi''|^{q_0}+|\zeta'|^{q_0})\psi_{\delta}(\xi'')d\xi''d\zeta'\\
&\leq& \int^{\xi+\delta}_{\zeta}(1+|\delta|^{q_0}+|\zeta'|^{q_0})\Big(\int_{\mathbb{R}}\psi_{\delta}(\xi'')d\xi''\Big)d\zeta'\\
&\leq& \int^{\xi+\delta}_{\zeta}(1+|\delta|^{q_0}+|\zeta'|^{q_0})d\zeta'\\
&\leq& C(q_0)(1+|\xi|^{q_0+1}+|\zeta|^{q_0+1}+|\delta|^{q_0+1}).
\end{eqnarray*}
Then, we deduce that
\begin{eqnarray*}
&&|\tilde{I}_1(t)|\\
&\leq& \varepsilon \mathcal{N}(q_0)C(q_0)\int^t_0\int_{(\mathbb{T}^N)^2}|\nabla_x \rho_{\gamma}(x-y)| \int_{\mathbb{R}^2}(1+|\xi|^{q_0+1}+|\zeta|^{q_0+1}+|\delta|^{q_0+1})d\nu^{1,\varepsilon}_{x,s}\otimes \nu^{2,\varepsilon}_{y,s}(\xi,\zeta)d\xi d\zeta dxdyds\\
&\leq& \varepsilon \gamma^{-1}\mathcal{N}(q_0)C(q_0)(1+|\delta|^{q_0+1})\\
&& +\varepsilon \gamma^{-1}\mathcal{N}(q_0)C(q_0)\Big(\underset{0\leq s\leq t}{{\rm{ess\sup}}}\ \|u^{\varepsilon}(s)\|^{q_0+1}_{L^{q_0+1}(\mathbb{T}^N)}+\underset{0\leq s\leq t}{{\rm{ess\sup}}}\ \|v^{\varepsilon}(s)\|^{q_0+1}_{L^{q_0+1}(\mathbb{T}^N)}\Big), \quad a.s..
\end{eqnarray*}
For $\tilde{I}_2(t)$, we have the same estimation as $\tilde{I}_1(t)$. Hence, we conclude that
\begin{eqnarray}\notag
|\tilde{I}(t)|&\leq& 2\varepsilon \gamma^{-1}\mathcal{N}(q_0)C(q_0)(1+|\delta|^{q_0+1})\\
\label{qq-29}
&&+2\varepsilon \gamma^{-1}\mathcal{N}(q_0)C(q_0)\Big(\underset{0\leq s\leq t}{{\rm{ess\sup}}}\ \|u^{\varepsilon}(s)\|^{q_0+1}_{L^{q_0+1}(\mathbb{T}^N)}+\underset{0\leq s\leq t}{{\rm{ess\sup}}}\ \|v^{\varepsilon}(s)\|^{q_0+1}_{L^{q_0+1}(\mathbb{T}^N)}\Big),\quad a.s..
\end{eqnarray}
By (\ref{equ-29}) in Hypothesis H, we arrive at
\begin{eqnarray*}
\tilde{J}(t)&=&\varepsilon\int^t_0\int_{(\mathbb{T}^N)^2}\int_{\mathbb{R}^2}\alpha \sum_{k\geq 1}|g_k(x,\xi)-g_k(y,\zeta)|^2d\nu^{1,\varepsilon}_{x,s}\otimes \nu^{2,\varepsilon}_{y,s}(\xi,\zeta)dxdyds\\
&\leq& \varepsilon D_1\int^t_0\int_{(\mathbb{T}^N)^2}\rho_{\gamma}(x-y)|x-y|^2\int_{\mathbb{R}^2}\psi_{\delta}(\xi-\zeta)d\nu^{1,\varepsilon}_{x,s}\otimes \nu^{2,\varepsilon}_{y,s}(\xi,\zeta)dxdyds\\
&& +\varepsilon D_1\int^t_0\int_{(\mathbb{T}^N)^2}\int_{\mathbb{R}^2}\rho_{\gamma}(x-y)\psi_{\delta}(\xi-\zeta)|\xi-\zeta|^2d\nu^{1,\varepsilon}_{x,s}\otimes \nu^{2,\varepsilon}_{y,s}(\xi,\zeta)dxdyds\\
&=:& \tilde{J}_{1}(t)+\tilde{J}_{2}(t).
\end{eqnarray*}
Note that
\begin{eqnarray*}
\int_{\mathbb{R}^2}\psi_{\delta}(\xi,\zeta)d\nu^{1,\varepsilon}_{x,s}\otimes \nu^{2,\varepsilon}_{y,s}(\xi,\zeta)&\leq& \delta^{-1}, \quad a.s.,
\\
\int_{(\mathbb{T}^N)^2}\rho_{\gamma}(x-y)|x-y|^2dxdy&\leq&\gamma^2,
\end{eqnarray*}
it follows that
\begin{eqnarray}\label{e-12}
\tilde{J}_{1}(t)\leq \varepsilon D_1\delta^{-1}\gamma^2, \quad a.s..
\end{eqnarray}
Referring to (35) in \cite{D-V-1}, it yields
\begin{eqnarray} \notag
\tilde{J}_{2}&\leq& \varepsilon \delta D_1 \int^t_0\int_{(\mathbb{T}^N)^2}\int_{|\xi-\zeta|\leq \delta}\rho_{\gamma}(x-y)\psi_{\delta}(\xi-\zeta)|\xi-\zeta|d\nu^{1,\varepsilon}_{x,s}\otimes \nu^{2,\varepsilon}_{y,s}(\xi,\zeta) dxdyds\\
\label{e-11}
&\leq& \varepsilon \delta D_1C_{\psi}, \quad a.s.,
\end{eqnarray}
where $C_{\psi}:=\sup_{\xi\in \mathbb{R}}\|\psi(\xi)\|$.
In view of (\ref{e-12}) and (\ref{e-11}), we arrive at
\begin{eqnarray*}
\tilde{J}(t)\leq \varepsilon D_1\delta^{-1}\gamma^2+\varepsilon D_1C_{\psi}\delta, \quad a.s..
\end{eqnarray*}
Combining all the above estimates, we conclude that
\begin{eqnarray}\notag
&&\int_{(\mathbb{T}^N)^2}\int_{\mathbb{R}^2}\rho_{\gamma} (x-y)\psi_{\delta}(\xi-\zeta)(f^{\pm}_1(x,t,\xi)\bar{f}^{\pm}_2(y,t,\zeta)+\bar{f}^{\pm}_1(x,t,\xi)f^{\pm}_2(y,t,\zeta))d\xi d\zeta dxdy\\
\notag
&\leq & \int_{(\mathbb{T}^N)^2}\int_{\mathbb{R}^2}\rho_{\gamma} (x-y)\psi_{\delta}(\xi-\zeta)(f_{1,0}(x,\xi)\bar{f}_{2,0}(y,\zeta)+\bar{f}_{1,0}(x,\xi)f_{2,0}(y,\zeta))d\xi d\zeta dxdy\\
\notag
&& +2\varepsilon \gamma^{-1}\mathcal{N}(q_0)C(q_0)(1+|\delta|^{q_0+1})+ \varepsilon D_1\delta^{-1}\gamma^2+\varepsilon D_1C_{\psi}\delta\\
\label{e-20}
&& +2\varepsilon \gamma^{-1}\mathcal{N}(q_0)C(q_0)\Big(\underset{0\leq s\leq t}{{\rm{ess\sup}}}\ \|u^{\varepsilon}(s)\|^{q_0+1}_{L^{q_0+1}(\mathbb{T}^N)}+\underset{0\leq s\leq t}{{\rm{ess\sup}}}\ \|v^{\varepsilon}(s)\|^{q_0+1}_{L^{q_0+1}(\mathbb{T}^N)}\Big)+\tilde{K}(t), \quad a.s..
\end{eqnarray}
For any $s\in (0,1)$, denote by
\[
R(s):=\int_{(\mathbb{T}^N)^2}\int_{\mathbb{R}^2}\rho_{\gamma} (x-y)\psi_{\delta}(\xi-\zeta)(f^{\pm}_1(x,s,\xi)\bar{f}^{\pm}_2(y,s,\zeta)+\bar{f}^{\pm}_1(x,s,\xi)f^{\pm}_2(y,s,\zeta))d\xi d\zeta dxdy.
\]
Then, we deduce from (\ref{e-20}) that
\begin{eqnarray*}
\underset{0\leq s\leq t}{{\rm{ess\sup}}}\ R(s)&\leq&  \int_{\mathbb{T}^N}\int_{\mathbb{R}}(f_{1,0}\bar{f}_{2,0}+\bar{f}_{1,0}f_{2,0})d\xi dx+\mathcal{E}_0(\gamma,\delta)\\
&&+2\varepsilon \gamma^{-1}\mathcal{N}(q_0)C(q_0)(1+|\delta|^{q_0+1})+ \varepsilon D_1\delta^{-1}\gamma^2+\varepsilon D_1C_{\psi}\delta\\
&& +2\varepsilon \gamma^{-1}\mathcal{N}(q_0)C(q_0)\Big(\underset{0\leq s\leq t}{{\rm{ess\sup}}}\ \|u^{\varepsilon}(s)\|^{q_0+1}_{L^{q_0+1}(\mathbb{T}^N)}+\underset{0\leq s\leq t}{{\rm{ess\sup}}}\ \|v^{\varepsilon}(s)\|^{q_0+1}_{L^{q_0+1}(\mathbb{T}^N)}\Big)\\
&& +\sup_{0\leq s\leq t}|\tilde{K}|(s), \quad a.s.,
\end{eqnarray*}
where $\lim_{\gamma,\delta\rightarrow 0}\mathcal{E}_0(\gamma,\delta)=0$.

Further, by H\"{o}lder inequality, it gives that
\begin{eqnarray}\notag
\left(\mathbb{E}\Big|\underset{0\leq s\leq t}{{\rm{ess\sup}}}\ R(s)\Big|^p\right)^{\frac{1}{p}}
&\lesssim & \int_{\mathbb{T}^N}\int_{\mathbb{R}}(f_{1,0}\bar{f}_{2,0}+\bar{f}_{1,0}f_{2,0})d\xi dx+\mathcal{E}_0(\gamma,\delta)\\
\notag
&& +2\varepsilon \gamma^{-1}\mathcal{N}(q_0)C(q_0)(1+|\delta|^{q_0+1})+ \varepsilon D_1\delta^{-1}\gamma^2+\varepsilon D_1C_{\psi}\delta\\
\label{e-1}
&& +2\varepsilon \gamma^{-1}\mathcal{N}(q_0)C(q_0)\mathcal{R} +\left(\mathbb{E}\Big|\sup_{s\in [0,t]}|\tilde{K}(s)|\Big|^p\right)^{\frac{1}{p}},
\end{eqnarray}
where
\begin{eqnarray*}
\mathcal{R}:=\sup_{0<\varepsilon\leq 1}\left\{\Big(\mathbb{E}\underset{0\leq s\leq 1}{{\rm{ess\sup}}}\ \|u^{\varepsilon}(s)\|^{p(q_0+1)}_{L^{p(q_0+1)}(\mathbb{T}^N)}\Big)^{\frac{1}{p}}+\Big(\mathbb{E}\underset{0\leq s\leq 1}{{\rm{ess\sup}}}\ \|v^{\varepsilon}(s)\|^{p(q_0+1)}_{L^{p(q_0+1)}(\mathbb{T}^N)}\Big)^{\frac{1}{p}}\right\}.
\end{eqnarray*}
Based on (\ref{qq-27}) and (\ref{qq-28}), we have
\begin{eqnarray}\label{qq-29-1}
\mathcal{R}<+\infty.
\end{eqnarray}

To estimate the stochastic integral term, we will use the following remarkable result from \cite{B-Y,Davis} that there exists a universal constant $C_0$ such that, for any $p\geq 2$ and for any continuous martingale $M_t$ with $M_0=0$,
\begin{eqnarray}\label{eq-10-1}
\mathbb{E}(|M^*_t|^p)\leq C^{\frac{p}{2}}_0 p^{\frac{p}{2}}\mathbb{E}\langle M\rangle^{\frac{p}{2}}_t,
\end{eqnarray}
where $M^*_t=\sup_{s\in [0,t]}|M_s|$.

Utilizing (\ref{eq-10-1}), we derive that
\begin{eqnarray}\label{e-7}
&&\mathbb{E}\Big|\sup_{s\in [0,t]}|\tilde{K}|(s)\Big|^p\\ \notag
&=& \varepsilon^{\frac{p}{2}}\mathbb{E}\Big|\sup_{s\in [0,t]}\sum_{k\geq 1}\int^s_0\int_{(\mathbb{T}^N)^2}\int_{\mathbb{R}^2}\chi_1(\xi,\zeta) \rho_{\gamma}(x-y)(g_k(x,\xi)-g_{k}(y,\zeta)) d \nu^{1,\varepsilon}_{x,r}\otimes \nu^{2,\varepsilon}_{y,r}(\xi,\zeta)dxdyd\beta_k(r)\Big|^p\\
\notag
&\leq& \varepsilon^{\frac{p}{2}}p^{\frac{p}{2}}C^{\frac{p}{2}}_0\mathbb{E}\Big[\int^t_0\sum_{k\geq 1}\Big|\int_{(\mathbb{T}^N)^2}\int_{\mathbb{R}^2}|g_k(x,\xi)-g_k(y,\zeta)|\rho_{\gamma}(x-y)\chi_1(\xi,\zeta) d \nu^{1,\varepsilon}_{x,r}\otimes \nu^{2,\varepsilon}_{y,r}(\xi,\zeta)dxdy\Big|^2dr\Big]^{\frac{p}{2}}.
\end{eqnarray}
Recall (\ref{e-6}) in Hypothesis H, it gives that
\[
|g_k(x,\xi)-g_k(y,\zeta)|\leq C^1_k(|x-y|+|\xi-\zeta|),\quad \sum_{k\geq 1}|C^1_k|^2\leq \frac{D_1}{2}:=D_2,
\]
hence, by (\ref{e-7}), we deduce that
\begin{eqnarray*}
&&\mathbb{E}\Big|\sup_{s\in [0,t]}|\tilde{K}|(s)\Big|^p\\
&\leq& \varepsilon^{\frac{p}{2}}p^{\frac{p}{2}}C^{\frac{p}{2}}_0D^{\frac{p}{2}}_2\mathbb{E}\Big[\int^t_0\Big|\int_{(\mathbb{T}^N)^2}\int_{\mathbb{R}^2}(|x-y|+|\xi-\zeta|)\rho_{\gamma}(x-y)\chi_1(\xi,\zeta) d \nu^{1,\varepsilon}_{x,r}\otimes \nu^{2,\varepsilon}_{y,r}(\xi,\zeta)dxdy\Big|^2dr\Big]^{\frac{p}{2}}.
\end{eqnarray*}
Since $\chi_1(\xi,\zeta)\leq 1$, it yields
\begin{eqnarray}\notag
\int_{(\mathbb{T}^N)^2}\int_{\mathbb{R}^2}|x-y|\rho_{\gamma}(x-y) \chi_1(\xi,\zeta)d \nu^{1,\varepsilon}_{x,r}\otimes \nu^{2,\varepsilon}_{y,r}(\xi,\zeta)dxdy\leq \gamma, \quad a.s..
\end{eqnarray}
Taking into account that $\nu^{1,\varepsilon}_{x,r}(\xi)=\delta_{u^{\varepsilon}(x,r)=\xi}$,  $\nu^{2,\varepsilon}_{y,r}(\zeta)=\delta_{v^{\varepsilon}(y,r)=\zeta}$, and by Corollary \ref{cor-1}, it follows that
\begin{eqnarray}\notag
&&\mathbb{E}\Big|\sup_{s\in [0,t]}|\tilde{K}|(s)\Big|^p\\ \notag
&\leq& \varepsilon^{\frac{p}{2}}p^{\frac{p}{2}}C^{\frac{p}{2}}_0D^{\frac{p}{2}}_2\mathbb{E}\Big[\int^t_0\Big|\gamma+\int_{(\mathbb{T}^N)^2}|u^{\varepsilon}-v^{\varepsilon}|\rho_{\gamma}(x-y) dxdy\Big|^2dr\Big]^{\frac{p}{2}}\\
\label{e-24}
&=&\varepsilon^{\frac{p}{2}}p^{\frac{p}{2}}C^{\frac{p}{2}}_0D^{\frac{p}{2}}_2\mathbb{E}\Big[\int^t_0\Big|\gamma+\int_{(\mathbb{T}^N)^2}|u^{\varepsilon,\pm}-v^{\varepsilon,\pm}|\rho_{\gamma}(x-y) dxdy\Big|^2dr\Big]^{\frac{p}{2}}.
\end{eqnarray}
With the help of the following identities
\begin{eqnarray}\label{e-4}
\int_{\mathbb{R}}I_{u^{\varepsilon,\pm}>\xi}\overline{I_{v^{\varepsilon,\pm}>\xi}}d\xi=(u^{\varepsilon,\pm}-v^{\varepsilon,\pm})^+,
\quad
\int_{\mathbb{R}}\overline{I_{u^{\varepsilon,\pm}>\xi}}I_{v^{\varepsilon,\pm}>\xi}d\xi=(u^{\varepsilon,\pm}-v^{\varepsilon,\pm})^-,
\end{eqnarray}
we deduce that
\begin{eqnarray}\notag
&& \int_{(\mathbb{T}^N)^2}|u^{\varepsilon,\pm}(x,r)-v^{\varepsilon,\pm}(y,r)|\rho_{\gamma}(x-y) dxdy\\ \notag
&=& \int_{(\mathbb{T}^N)^2}\Big((u^{\varepsilon,\pm}(x,r)-v^{\varepsilon,\pm}(y,r))^++(u^{\varepsilon,\pm}(x,r)-v^{\varepsilon,\pm}(y,r))^{-}\Big)\rho_{\gamma}(x-y) dxdy\\
\notag
&=&\int_{(\mathbb{T}^N)^2}\int_{\mathbb{R}}(\bar{f}^{\pm}_1(x,r,\xi)f^{\pm}_2(y,r,\xi)+f^{\pm}_1(x,r,\xi)\bar{f}^{\pm}_2(y,r,\xi))\rho_{\gamma}(x-y)d \xi  dxdy\\ \notag
&\leq& 4\delta+\int_{(\mathbb{T}^N)^2}\int_{\mathbb{R}^2}(\bar{f}^{\pm}_1(x,r,\xi)f^{\pm}_2(y,r,\zeta)+f^{\pm}_1(x,r,\xi)\bar{f}^{\pm}_2(y,r,\zeta))\rho_{\gamma}(x-y)\psi_{\delta}(\xi-\zeta) d \xi d\zeta dxdy\\
\label{e-25}
&=& 4\delta+R(r), \quad a.s.,
\end{eqnarray}
where we have used (\ref{e-23}) and (\ref{e-22}).
Combining (\ref{e-24}) and (\ref{e-25}), we deduce that for $0< t<1$, it holds that
\begin{eqnarray}\notag
\mathbb{E}\Big|\sup_{s\in [0,t]}|\tilde{K}(s)|\Big|^p
&\leq& \varepsilon^{\frac{p}{2}}p^{\frac{p}{2}}C^{\frac{p}{2}}_0D^{\frac{p}{2}}_2\mathbb{E}\Big[\int^t_0\Big|\gamma+4\delta+R(r)\Big|^2dr\Big]^{\frac{p}{2}}\\
\notag
&\leq & \varepsilon^{\frac{p}{2}}p^{\frac{p}{2}}C^{\frac{p}{2}}_0D^{\frac{p}{2}}_22^{\frac{p}{2}}\mathbb{E}\Big[\int^t_0\Big|\gamma+4\delta\Big|^2dr+\int^t_0|R(r)|^2dr\Big]^{\frac{p}{2}}
\\
\label{e-2}
&\leq &
\varepsilon^{\frac{p}{2}}p^{\frac{p}{2}}C^{\frac{p}{2}}_0D^{\frac{p}{2}}_2 2^p|\gamma+4\delta|^p+\varepsilon^{\frac{p}{2}}p^{\frac{p}{2}}C^{\frac{p}{2}}_0D^{\frac{p}{2}}_22^p\mathbb{E}\Big(\int^t_0R^2(r)dr\Big)^{\frac{p}{2}}.
\end{eqnarray}
Then, it follows  from (\ref{e-1}) and (\ref{e-2}) that
\begin{eqnarray*}\notag
&&\left(\mathbb{E}\Big|\underset{0\leq s\leq t}{{\rm{ess\sup}}}\ R(s)\Big|^p\right)^{\frac{1}{p}}\\
&\lesssim & \int_{\mathbb{T}^N}\int_{\mathbb{R}}(f_{1,0}\bar{f}_{2,0}+\bar{f}_{1,0}f_{2,0})d\xi dx+\mathcal{E}_0(\gamma,\delta)+2\varepsilon \gamma^{-1}\mathcal{N}(q_0)C(q_0)(1+|\delta|^{q_0+1})\\
&&+ \varepsilon D_1\delta^{-1}\gamma^2+\varepsilon D_1C_{\psi}\delta+2\varepsilon \gamma^{-1}\mathcal{N}(q_0)C(q_0)\mathcal{R}+\varepsilon^{\frac{1}{2}}p^{\frac{1}{2}}C^{\frac{1}{2}}_0D^{\frac{1}{2}}_2
|\gamma+4\delta|\\ \notag
&& +\varepsilon^{\frac{1}{2}}p^{\frac{1}{2}}C^{\frac{1}{2}}_0D^{\frac{1}{2}}_2\Big(\mathbb{E}\Big(\int^t_0R^2(r)dr\Big)^{\frac{p}{2}}\Big)^{\frac{1}{p}}.
\end{eqnarray*}
For any $p\geq 2$, by Minkowski's integral inequality, it holds that
\begin{eqnarray*}
\left(\mathbb{E}\Big[\int^t_0R^2(r)dr\Big]^{\frac{p}{2}}\right)^{\frac{1}{p}}
&=& \Big[\Big(\mathbb{E}(\int^t_0R^2(r)dr)^{\frac{p}{2}}\Big)^{\frac{2}{p}}\Big]^{\frac{1}{2}}\\
&\leq& \Big[\int^t_0\Big(\mathbb{E}R^p(r)\Big)^{\frac{2}{p}}dr\Big]^{\frac{1}{2}}\\
&\leq& \Big[\int^t_0\Big(\mathbb{E}\Big|\underset{0\leq s\leq r}{{\rm{ess\sup}}}\ R(s)\Big|^p\Big)^{\frac{2}{p}}dr\Big]^{\frac{1}{2}}.
\end{eqnarray*}
Thus, we reach
\begin{eqnarray}\notag
&&\Big(\mathbb{E}|\underset{0\leq s\leq t}{{\rm{ess\sup}}}\ R(s)|^p\Big)^{\frac{2}{p}}\\ \notag
&\leq & \mathcal{D}^2 \Big[\int_{\mathbb{T}^N}\int_{\mathbb{R}}(f_{1,0}\bar{f}_{2,0}+\bar{f}_{1,0}f_{2,0})d\xi dx+\mathcal{E}_0(\gamma,\delta)+2\varepsilon \gamma^{-1}\mathcal{N}(q_0)C(q_0)(1+|\delta|^{q_0+1})\\ \notag
&&+ \varepsilon D_1\delta^{-1}\gamma^2+\varepsilon D_1C_{\psi}\delta+2\varepsilon \gamma^{-1}\mathcal{N}(q_0)C(q_0)\mathcal{R}+\varepsilon^{\frac{1}{2}}p^{\frac{1}{2}}C^{\frac{1}{2}}_0D^{\frac{1}{2}}_2
|\gamma+4\delta|\Big]^2\\
\label{e-8}
&& +\mathcal{D}^2 \varepsilon pC_0D_2\int^t_0\left(\mathbb{E}\Big|\underset{0\leq s\leq r}{{\rm{ess\sup}}}\ R(s)\Big|^p\right)^{\frac{2}{p}}dr,
\end{eqnarray}
{\color{r}where $\mathcal{D}$ is defined in section 2.}
Let $G(t):=\left(\mathbb{E}\Big|\underset{0\leq s\leq t}{{\rm{ess\sup}}}\ R(s)\Big|^p\right)^{\frac{2}{p}}$,
applying Gronwall inequality to (\ref{e-8}), we get
\begin{eqnarray}\notag
G(t)&\leq& \mathcal{D}^2e^{\mathcal{D}^2\varepsilon pC_0D_2}\Big[\int_{\mathbb{T}^N}\int_{\mathbb{R}}(f_{1,0}\bar{f}_{2,0}+\bar{f}_{1,0}f_{2,0})d\xi dx+\mathcal{E}_0(\gamma,\delta)+2\varepsilon \gamma^{-1}\mathcal{N}(q_0)C(q_0)(1+|\delta|^{q_0+1})\\
\label{e-9}
&&  \quad \quad + \varepsilon D_1\delta^{-1}\gamma^2+\varepsilon D_1C_{\psi}\delta+2\varepsilon \gamma^{-1}\mathcal{N}(q_0)C(q_0)\mathcal{R}+\varepsilon^{\frac{1}{2}}p^{\frac{1}{2}}C^{\frac{1}{2}}_0D^{\frac{1}{2}}_2
|\gamma+4\delta|\Big]^2,
\end{eqnarray}
which implies that
\begin{eqnarray}\notag
&&\left(\mathbb{E}\Big|\underset{0\leq s\leq 1}{{\rm{ess\sup}}}\ R(s)\Big|^p\right)^{\frac{1}{p}}\\ \notag
&\lesssim& e^{\mathcal{D}^2\varepsilon pC_0D_2}\Big[\int_{\mathbb{T}^N}\int_{\mathbb{R}}(f_{1,0}\bar{f}_{2,0}+\bar{f}_{1,0}f_{2,0})d\xi dx+\mathcal{E}_0(\gamma,\delta)+2\varepsilon \gamma^{-1}\mathcal{N}(q_0)C(q_0)(1+|\delta|^{q_0+1})\\
\label{qq-10}
&& +\varepsilon D_1\delta^{-1}\gamma^2+\varepsilon D_1C_{\psi}\delta+2\varepsilon \gamma^{-1}\mathcal{N}(q_0)C(q_0)\mathcal{R}+\varepsilon^{\frac{1}{2}}p^{\frac{1}{2}}C^{\frac{1}{2}}_0D^{\frac{1}{2}}_2
|\gamma+4\delta|\Big].
\end{eqnarray}
Recall the definition of $R(s)$, it holds that
 \begin{eqnarray}\notag
&&\left(\mathbb{E}\Big|\underset{0\leq s\leq 1}{{\rm{ess\sup}}}\ \int_{(\mathbb{T}^N)^2}\int_{(\mathbb{R})^2}\rho_{\gamma}(x-y)\psi_{\gamma}(\xi-\zeta)(f^{\pm}_1(s,x,\xi)\bar{f}^{\pm}_2(s,y,\zeta)+\bar{f}^{\pm}_1(s,x,\xi){f}^{\pm}_2(s,y,\zeta))d\xi d\zeta dxdy\Big|^p\right)^{\frac{1}{p}}\\ \notag
&\lesssim& e^{\mathcal{D}^2\varepsilon pC_0D_2}\Big[\int_{\mathbb{T}^N}\int_{\mathbb{R}}(f_{1,0}\bar{f}_{2,0}+\bar{f}_{1,0}f_{2,0})d\xi dx+\mathcal{E}_0(\gamma,\delta)+2\varepsilon \gamma^{-1}\mathcal{N}(q_0)C(q_0)(1+|\delta|^{q_0+1})\\
\label{qq-11}
&&+ \varepsilon D_1\delta^{-1}\gamma^2+\varepsilon D_1C_{\psi}\delta+2\varepsilon \gamma^{-1}\mathcal{N}(q_0)C(q_0)\mathcal{R}+\varepsilon^{\frac{1}{2}}p^{\frac{1}{2}}C^{\frac{1}{2}}_0D^{\frac{1}{2}}_2
|\gamma+4\delta|\Big].
\end{eqnarray}
Applying the same procedure to $f^{\pm}_2$ and $\bar{f}^{\pm}_2$ (in this case, $A=0$ and $\int_{\mathbb{T}^N}\int_{\mathbb{R}}(f_{2,0}\bar{f}_{2,0}+\bar{f}_{2,0}f_{2,0})d\xi dx=0$), we obtain
\begin{eqnarray*}
&&\left(\mathbb{E}\Big|\underset{0\leq s\leq 1}{{\rm{ess\sup}}}\ \int_{(\mathbb{T}^N)^2}\int_{(\mathbb{R})^2}\rho_{\gamma}(x-y)\psi_{\gamma}(\xi-\zeta)(f^{\pm}_2(s,x,\xi)\bar{f}^{\pm}_2(s,y,\zeta)+\bar{f}^{\pm}_2(s,x,\xi){f}^{\pm}_2(s,y,\zeta))d\xi d\zeta dxdy\Big|^p\right)^{\frac{1}{p}}\\
&\lesssim& e^{\mathcal{D}^2\varepsilon pC_0D_2}\Big[\mathcal{E}_0(\gamma,\delta)+\varepsilon D_1\delta^{-1}\gamma^2+\varepsilon D_1C_{\psi}\delta+\varepsilon^{\frac{1}{2}}p^{\frac{1}{2}}C^{\frac{1}{2}}_0D^{\frac{1}{2}}_2
|\gamma+4\delta|\Big].
\end{eqnarray*}
For the sake of convenience, denote by
\[
Q(s):=\int_{(\mathbb{T}^N)^2}\int_{(\mathbb{R})^2}\rho_{\gamma}(x-y)\psi_{\gamma}(\xi-\zeta)(f^{\pm}_2(s,x,\xi)\bar{f}^{\pm}_2(s,y,\zeta)+\bar{f}^{\pm}_2(s,x,\xi){f}^{\pm}_2(s,y,\zeta))d\xi d\zeta dxdy,
\]
then, it yields
\begin{eqnarray}\label{qq-12}
\left(\mathbb{E}\Big|\underset{0\leq s\leq 1}{{\rm{ess\sup}}}\ Q(s)\Big|^p\right)^{\frac{1}{p}}
\lesssim e^{\mathcal{D}^2\varepsilon pC_0D_2}\Big[\mathcal{E}_0(\gamma,\delta)+\varepsilon D_1\delta^{-1}\gamma^2+\varepsilon D_1C_{\psi}\delta+\varepsilon^{\frac{1}{2}}p^{\frac{1}{2}}C^{\frac{1}{2}}_0D^{\frac{1}{2}}_2
|\gamma+4\delta|\Big].
\end{eqnarray}
On the other hand, from (\ref{qq-3}), it follows that
\begin{eqnarray}\notag
&&\left(\mathbb{E} \Big|\underset{0\leq s\leq 1}{{\rm{ess\sup}}}\ \int_{\mathbb{T}^N}\int_{\mathbb{R}}(f^{\pm}_1(s,x,\xi)\bar{f}^{\pm}_2(s,x,\xi)+\bar{f}^{\pm}_1(s,x,\xi){f}^{\pm}_2(s,x,\xi))d\xi dx\Big|^p\right)^{\frac{1}{p}}\\ \label{eee-1}
&\lesssim& \left(\mathbb{E}\Big|\underset{0\leq s\leq 1}{{\rm{ess\sup}}}\ |\mathcal{E}_s(\gamma,\delta)|\Big|^p\right)^{\frac{1}{p}}+\left(\mathbb{E}\Big|\underset{0\leq s\leq 1}{{\rm{ess\sup}}}\ R(s)\Big|^p\right)^{\frac{1}{p}}.
\end{eqnarray}
{\color{rr}
In the following, we devote to making estimates of $ \left(\mathbb{E}\Big|\underset{0\leq s\leq 1}{{\rm{ess\sup}}}\ |\mathcal{E}_s(\gamma,\delta)|\Big|^p\right)^{\frac{1}{p}}$.
For any $s\in (0,1)$, we have
\begin{eqnarray*}
\mathcal{E}_s(\gamma, \delta)&=&\int_{(\mathbb{T}^N)^2}\int_{\mathbb{R}^2}(f^{\pm}_1(x,s,\xi)\bar{f}^{\pm}_2(y,s,\zeta)+\bar{f}^{\pm}_1(x,s,\xi){f}^{\pm}_2(y,s,\zeta))\rho_{\gamma}(x-y)\psi_{\delta}(\xi-\zeta)dxdyd\xi d\zeta
\\
&&-\int_{\mathbb{T}^N}\int_{\mathbb{R}}(f^{\pm}_1(x,s,\xi)\bar{f}^{\pm}_2(x,s,\xi)+\bar{f}^{\pm}_1(x,s,\xi){f}^{\pm}_2(x,s,\xi))d\xi dx\\
&=& \Big[\int_{(\mathbb{T}^N)^2}\int_{\mathbb{R}}\rho_{\gamma}(x-y)(f^{\pm}_1(x,s,\xi)\bar{f}^{\pm}_2(y,s,\xi)+\bar{f}^{\pm}_1(x,s,\xi){f}^{\pm}_2(y,s,\xi))d\xi dxdy\\
&& -\int_{\mathbb{T}^N}\int_{\mathbb{R}}(f^{\pm}_1(x,s,\xi)\bar{f}^{\pm}_2(x,s,\xi)+\bar{f}^{\pm}_1(x,s,\xi){f}^{\pm}_2(x,s,\xi))d\xi dx\Big]\\
&& +\Big[\int_{(\mathbb{T}^N)^2}\int_{\mathbb{R}^2}(f^{\pm}_1(x,s,\xi)\bar{f}^{\pm}_2(y,s,\zeta)+\bar{f}^{\pm}_1(x,s,\xi){f}^{\pm}_2(y,s,\zeta))\rho_{\gamma}(x-y)\psi_{\delta}(\xi-\zeta)dxdyd\xi d\zeta\\
&& -\int_{(\mathbb{T}^N)^2}\int_{\mathbb{R}}\rho_{\gamma}(x-y)(f^{\pm}_1(x,s,\xi)\bar{f}^{\pm}_2(y,s,\xi)+\bar{f}^{\pm}_1(x,s,\xi){f}^{\pm}_2(y,s,\xi))d\xi dxdy \Big]\\
&=:&H_1+H_2,
\end{eqnarray*}
By (\ref{e-23}) and (\ref{e-22}), it gives
\begin{eqnarray}\label{qq-15}
|H_2|\leq 4\delta, \quad a.s..
\end{eqnarray}
Moreover, it is easy to deduce that
\begin{eqnarray*}
|H_1|&\leq& \Big|\int_{(\mathbb{T}^N)^2}\rho_{\gamma}(x-y)\int_{\mathbb{R}}I_{\bar{u}^{\varepsilon,\pm}(x,s)>\xi}(I_{v^{\varepsilon,\pm}(x,s)\leq \xi}-I_{v^{\varepsilon}(y,s)\leq \xi})d\xi dxdy\Big|\\
&& +\Big|\int_{(\mathbb{T}^N)^2}\rho_{\gamma}(x-y)\int_{\mathbb{R}}I_{\bar{u}^{\varepsilon,\pm}(x,s)\leq\xi}(I_{v^{\varepsilon,\pm}(x,s)> \xi}-I_{v^{\varepsilon,\pm}(y,s)> \xi})d\xi dxdy\Big|\\
&\leq& 2\int_{(\mathbb{T}^N)^2}\rho_{\gamma}(x-y)|v^{\varepsilon,\pm}(x,s)-v^{\varepsilon,\pm}(y,s)|dxdy, \quad a.s..
\end{eqnarray*}
Utilizing (\ref{e-23}) and (\ref{e-22}) again, it follows that
\begin{eqnarray*}
&&\int_{(\mathbb{T}^N)^2}\rho_{\gamma}(x-y)|v^{\varepsilon,\pm}(x,s)-v^{\varepsilon,\pm}(y,s)|dxdy\\
&=& \int_{(\mathbb{T}^N)^2}\int_{\mathbb{R}}\rho_{\gamma}(x-y)(f^{\pm}_2(x,s,\xi)\bar{f}^{\pm}_2(y,s,\xi)+\bar{f}^{\pm}_2(x,s,\xi){f}^{\pm}_2(y,s,\xi))d\xi dxdy\\
&\leq& \int_{(\mathbb{T}^N)^2}\int_{\mathbb{R}^2}\rho_{\gamma}(x-y)\psi_{\delta}(\xi-\zeta)(f^{\pm}_2(x,s,\xi)\bar{f}^{\pm}_2(y,s,\zeta)+\bar{f}^{\pm}_2(x,s,\xi){f}^{\pm}_2(y,s,\zeta))d\xi d\zeta dxdy+4\delta\\
&=&Q(s)+4\delta, \quad a.s..
\end{eqnarray*}
 Then,
\begin{eqnarray}\label{qq-14}
|H_1|\leq 2Q(s)+8\delta, \quad a.s..
\end{eqnarray}
Collecting (\ref{qq-15}) and (\ref{qq-14}), it yields
\begin{eqnarray*}
|\mathcal{E}_s(\gamma, \delta)|\leq 2Q(s)+12\delta, \quad a.s.,
\end{eqnarray*}
hence, by (\ref{qq-12}), we deduce that
\begin{eqnarray}\notag
&&\Big(\mathbb{E}\big|\underset{0\leq s\leq 1}{{\rm{ess\sup}}}\ |\mathcal{E}_s(\gamma,\delta)|\big|^p\Big)^{\frac{1}{p}}\\ \notag
&\lesssim& \Big(\mathbb{E}|\underset{0\leq s\leq 1}{{\rm{ess\sup}}}\ Q(s)|^p\Big)^{\frac{1}{p}}+\delta\\
\label{qq-16}
&\lesssim&e^{\mathcal{D}^2\varepsilon pC_0D_2}\Big[\mathcal{E}_0(\gamma,\delta)+\varepsilon D_1\delta^{-1}\gamma^2+\varepsilon D_1C_{\psi}\delta+\varepsilon^{\frac{1}{2}}p^{\frac{1}{2}}C^{\frac{1}{2}}_0D^{\frac{1}{2}}_2
|\gamma+4\delta|\Big]+\delta.
\end{eqnarray}
}
Combining (\ref{qq-10}) and (\ref{qq-16}), we deduce from (\ref{eee-1}) that
\begin{eqnarray*}\notag
&&\left(\mathbb{E} \Big|\underset{0\leq s\leq 1}{{\rm{ess\sup}}}\ \int_{\mathbb{T}^N}\int_{\mathbb{R}}(f^{\pm}_1(s,x,\xi)\bar{f}^{\pm}_2(s,x,\xi)+\bar{f}^{\pm}_1(s,x,\xi){f}^{\pm}_2(s,x,\xi))d\xi dx\Big|^p\right)^{\frac{1}{p}}\\ \notag
&\lesssim&e^{\mathcal{D}^2\varepsilon pC_0D_2}\Big[\int_{\mathbb{T}^N}\int_{\mathbb{R}}(f_{1,0}\bar{f}_{2,0}+\bar{f}_{1,0}f_{2,0})d\xi dx+2\mathcal{E}_0(\gamma,\delta)+2\varepsilon \gamma^{-1}\mathcal{N}(q_0)C(q_0)(1+|\delta|^{q_0+1})\\
&& +2\varepsilon D_1\delta^{-1}\gamma^2+2\varepsilon D_1C_{\psi}\delta+2\varepsilon \gamma^{-1}\mathcal{N}(q_0)C(q_0)\mathcal{R}+2\varepsilon^{\frac{1}{2}}p^{\frac{1}{2}}C^{\frac{1}{2}}_0D^{\frac{1}{2}}_2
|\gamma+4\delta|\Big]+\delta.
\end{eqnarray*}
Note that we have $f^{\pm}_1=I_{u^{\varepsilon,\pm}>\xi}$ and $f^{\pm}_2=I_{v^{\varepsilon,\pm}>\xi}$ with initial data $f_{1,0}=I_{\eta>\xi}$ and ${f}_{2,0}=I_{\eta>\xi}$, respectively. With the help of (\ref{e-4}),
we deduce that
\begin{eqnarray}\label{e-26}
\left(\mathbb{E}\Big|\underset{0\leq s\leq 1}{{\rm{ess\sup}}}\ \|u^{\varepsilon,\pm}(s)-v^{\varepsilon,\pm}(s)\|_{L^1(\mathbb{T}^N)}\Big|^p\right)^{\frac{1}{p}}
\lesssim r(\varepsilon,p, \gamma, \delta),
\end{eqnarray}
where
\begin{eqnarray}\notag
&&r(\varepsilon,p, \gamma, \delta)\\ \notag
&:=&e^{\mathcal{D}^2\varepsilon pC_0D_2}\Big[\int_{\mathbb{T}^N}\int_{\mathbb{R}}(f_{1,0}\bar{f}_{2,0}+\bar{f}_{1,0}f_{2,0})d\xi dx+2\mathcal{E}_0(\gamma,\delta)+2\varepsilon \gamma^{-1}\mathcal{N}(q_0)C(q_0)(1+|\delta|^{q_0+1})\\
\label{e-10}
&& +2\varepsilon D_1\delta^{-1}\gamma^2+2\varepsilon D_1C_{\psi}\delta+2\varepsilon \gamma^{-1}\mathcal{N}(q_0)C(q_0)\mathcal{R}+2\varepsilon^{\frac{1}{2}}p^{\frac{1}{2}}C^{\frac{1}{2}}_0D^{\frac{1}{2}}_2
|\gamma+4\delta|\Big]+\delta.
\end{eqnarray}
Taking
\[
\delta=\gamma=\varepsilon^{\frac{1}{2}},
\]
and
letting $p=\frac{1}{\varepsilon}$, by (\ref{qq-5}) and (\ref{qq-29-1}), we have
\begin{eqnarray*}
r(\varepsilon,p, \gamma, \delta)&=&e^{\mathcal{D}^2C_0D_2} \Big[2\mathcal{E}_0(\gamma,\delta)+2\mathcal{N}(q_0)C(q_0) \varepsilon^{\frac{1}{2}}(1+\varepsilon^{\frac{q_0+1}{2}})+2D_1\varepsilon^{\frac{3}{2}} +2D_1C_{\psi}\varepsilon^{\frac{3}{2}}\\
&& +2\varepsilon^{\frac{1}{2}}\mathcal{N}(q_0)C(q_0)\mathcal{R}+10\varepsilon^{\frac{1}{2}}C^{\frac{1}{2}}_0D^{\frac{1}{2}}_2\Big]+\varepsilon^{\frac{1}{2}}\\
&\rightarrow& 0, \quad {\rm{as}}\ \ \varepsilon\rightarrow 0.
\end{eqnarray*}
Therefore, we deduce from (\ref{e-26}) that
\begin{eqnarray}\notag
\left(\mathbb{E}\Big|\underset{0\leq s\leq 1}{{\rm{ess\sup}}}\ \|u^{\varepsilon,\pm}(s)-v^{\varepsilon,\pm}(s)\|_{L^1(\mathbb{T}^N)}\Big|^p\right)^{\frac{1}{p}}
\rightarrow 0, \quad {\rm{as}}\ \ \varepsilon\rightarrow 0.
\end{eqnarray}
By using Chebyshev inequality and (\ref{e-26}), for any $\iota>0$, we deduce that
\begin{eqnarray*}
&&\varepsilon \log \mathbb{P}\Big(\|u^{\varepsilon}-v^{\varepsilon}\|_{L^1([0,1];L^1(\mathbb{T}^N))}>\iota\Big)\\
&\leq& \varepsilon \log \Big[\mathbb{E}\Big( \|u^{\varepsilon,\pm}-v^{\varepsilon,\pm}\|^p_{L^1([0,1];L^1(\mathbb{T}^N))}\Big)/{\iota}^p\Big]\\
&\leq& -\log \iota+\log\left[\left(\mathbb{E} \Big|\underset{0\leq s\leq 1}{{\rm{ess\sup}}}\ \|u^{\varepsilon,\pm}(s)-v^{\varepsilon,\pm}(s)\|_{L^1(\mathbb{T}^N)}\Big|^p\right)^{\frac{1}{p}}\right]\\
&\rightarrow& -\infty, \quad {\rm{as}}\ \ \varepsilon\rightarrow 0.
\end{eqnarray*}
We complete the proof.
\end{proof}
%
%
%
%
%

\noindent{\bf  Acknowledgements}\quad This work is partly supported by National Natural Science Foundation of China (No. 11931004,11801032,11971227), Key Laboratory of Random Complex Structures and Data Science, Academy of Mathematics and Systems Science, Chinese Academy of Sciences (No. 2008DP173182) and Beijing Institute of Technology Research Fund Program for Young Scholars.

\def\refname{ References}

\end{document}